\documentclass[11pt]{article}
\usepackage{amsfonts}
\usepackage{stmaryrd}
\usepackage{graphicx}
\usepackage{epsfig}
\usepackage{xcolor}
\usepackage{amsmath}
\usepackage{amsthm}
\usepackage{amssymb}
\usepackage{mathrsfs}
\usepackage{float}
\usepackage{multirow}
\usepackage{verbatim}
\usepackage{fancyhdr}
\usepackage{subcaption}
\usepackage{subcaption}
\usepackage{color}
\usepackage{mathtools}
\usepackage{mathrsfs}
\usepackage{sectsty}
\usepackage[title]{appendix}
\usepackage{threeparttable}
\usepackage{dcolumn}
\usepackage{booktabs}
\usepackage{indentfirst}
\usepackage{setspace}
\usepackage{bm}
\usepackage{enumerate}
\usepackage{lineno}
\usepackage{bbm}
\usepackage{tikz}
\usetikzlibrary{patterns, arrows.meta, decorations.pathreplacing, intersections}
\usepackage[labelfont=bf,labelsep=period]{caption}
\newcolumntype{d}[1]{D{.}{.}{#1}}
\usepackage{hyperref}
\hypersetup{colorlinks,linkcolor=blue,citecolor=blue}

\newtheorem{theorem}{Theorem}[section]
\newtheorem{example}{Example}[section]

\newtheorem{remark}{Remark}[section]

\allowdisplaybreaks

\usepackage{graphicx}

\begin{document}

\pagenumbering{arabic}
\baselineskip=1.3pc

\vspace*{0.5in}

\begin{center}

{\Large{\bf A cut energy-based discontinuous Galerkin method for the wave equation}}

\end{center}

\vspace{.03in}

\centerline{
Yangxin Fu\footnote{School of Mathematical Sciences,
         University of Science and Technology of China,
         Hefei, Anhui 230026, P.R. China. 
         E-mail: yxfu@mail.ustc.edu.cn}, \ 
 Yan Jiang\footnote{School of Mathematical Sciences,
         University of Science and Technology of China, Hefei,
         Anhui 230026, P.R. China.  
         E-mail: jiangy@ustc.edu.cn},\ 
Siyang Wang\footnote{Corresponding author. 
    Department of Mathematics and Mathematical Statistics, Umeå University, Umeå 90187, Sweden. Email: siyang.wang@umu.se
}, \ 
Lu Zhang\footnote{Department of Computational Applied Mathematics and Operations Research, Rice University, Houston, TX
77005, USA. Email: lz82@rice.edu
}
}

\vspace{.1in}

\begin{abstract}
This paper develops a high-order cut energy-based discontinuous Galerkin (CutEDG) method for second-order acoustic wave equations on complex domains and across stationary material interfaces. The method employs unfitted Cartesian meshes and addresses the small-cut-cell problem by introducing ghost-penalty stabilization directly into the local bilinear forms defining the EDG energy, while leaving the original spatial coupling and numerical-flux structure unchanged. We establish semidiscrete energy stability and a priori error estimates showing that high-order accuracy is retained. The formulation is further extended to interface problems with discontinuous wave speeds. Numerical experiments confirm the predicted convergence rates and demonstrate robustness with respect to arbitrarily small cut cells. In particular, the explicit time-step restriction is governed by the background mesh size rather than the smallest physical cut cell, and the stabilized mass matrix exhibits cut-independent conditioning.
\end{abstract}



% \linenumbers

\section{Introduction}

Wave propagation in geometrically complex media arises in a broad range of applications, including acoustics, seismic modeling, nondestructive testing, and wave-structure interaction; see, for example, \cite{ihlenburg1998finite,komatitsch1999introduction,schmerr2016fundamentals} and the references therein. In many of these applications, waves propagate over many wavelengths or over long time intervals, so that numerical dispersion and phase errors can accumulate significantly. High-order discretizations are therefore particularly attractive, since they can achieve a prescribed accuracy with fewer degrees of freedom per wavelength than low-order methods \cite{ainsworth2004discrete}. At the same time, practical computational domains may contain curved boundaries, embedded obstacles, or material interfaces with complicated geometry. Developing numerical methods that combine high-order accuracy with geometric flexibility is thus an important issue in computational wave propagation.

Conventional high-order finite element and discontinuous Galerkin (DG) methods typically rely on meshes that conform to the physical boundary or interface; see, e.g., \cite{sevilla2011comparison,zhang2016curved,kawecki2020finite}. Although body-fitted and curvilinear meshes are effective when a suitable mesh is available, their construction may become costly and technically demanding for complicated geometries, especially at high order. This difficulty becomes more pronounced when many geometric configurations have to be considered or when interfaces vary between simulations. Unfitted and immersed methods provide an attractive alternative \cite{muller2017high, qin2013discontinuous,bastian2009unfitted} by embedding the physical domain in a simple background mesh, often structured or Cartesian, while allowing physical boundaries and interfaces to intersect the mesh without requiring mesh conformity.

The flexibility of an unfitted mesh, however, comes with the well-known small-cut-cell problem \cite{burman2010ghost,burman2012fictitious,burman2015cutfem}. When a boundary or interface passes arbitrarily close to a mesh vertex or edge, a cut element inside the domain can become arbitrarily small. As a consequence, discrete functions supported on such elements may receive only weak control from integration over the physical domain, leading to severe ill-conditioning and loss of uniform stability. For time-dependent wave problems, the issue is particularly important because arbitrarily small cut cells may also impose a prohibitively restrictive time-step constraint for explicit time integration. An effective unfitted wave discretization must therefore remain robust with respect to the cut configuration, including the limiting case in which the physical fraction of a cut element tends to zero.

Ghost-penalty stabilization provides a systematic mechanism for recovering such robustness and is a central ingredient of Cut Finite Element Methods (CutFEM) \cite{burman2012fictitious,burman2015cutfem, burman2025cut}. Stabilized CutFEM formulations have been developed for elliptic, interface, and wave problems \cite{burman2012fictitious,burman2015cutfem,sticko2016stabilized}. For wave equations, higher-order extensions have demonstrated high-order accuracy and favorable explicit stability properties on unfitted meshes \cite{sticko2019higher,Weber2026}; in particular, \cite{sticko2019higher} shows numerically that the CFL restriction is essentially unaffected by the immersed boundary. Related ideas have also been extended to discontinuous finite element spaces, for example, \cite{Gurkan2019} developed a stabilized CutDG framework for elliptic boundary value and interface problems, with error and condition-number estimates that are robust with respect to the cut configuration. Ghost-penalty-based CutDG formulations have subsequently been studied for hyperbolic problems \cite{Gurkan2020,fu2021high,fu2022high,fu2024bound}.
In particular, high-order CutDG methods for hyperbolic conservation laws have shown that explicit time stepping can be performed with a time-step restriction independent of the size of the smallest cut element \cite{fu2021high,fu2024bound}. These developments establish ghost stabilization as an effective tool for unfitted discretizations, but they do not by themselves determine how the stabilization should be incorporated into a particular wave formulation.

In this work, we consider this question within the energy-based discontinuous Galerkin (EDG) framework \cite{appelo2015new}. DG methods are attractive for high-order wave propagation because of their local structure, high-order accuracy, and natural compatibility with explicit time integration \cite{hesthaven2008nodal, cockburn1999discontinuous, grote2006discontinuous, cockburn2001runge}. Among the different DG formulations available for second-order wave equations (see, e.g., \cite{xing2013energy,grote2006discontinuous,shukla2022space} and reference therein), EDG is particularly well suited to the objective of the present work because its construction is organized directly around a discrete counterpart of the wave energy (see, e.g.,  \cite{appelo2015new,appelo2018energy,appelo2020energy,zhang2019energy,ren2024energy}). The mesh independent numerical fluxes can be chosen so that the semidiscrete scheme satisfies an energy-conserving or energy-dissipating relation. Thus, the energy is not only a tool for analyzing the final scheme, but also provides a natural principle for designing the discretization. This point becomes important on an unfitted mesh. In the EDG formulation, the discrete energy is associated with two local bilinear forms, one involving the gradient of the displacement variable and the other the $L^2$ inner product of the velocity variable. When the physical portion of an element becomes arbitrarily small, both forms lose uniform control of the degrees of freedom supported on that element. The small-cut problem therefore acts directly on the local forms that define the EDG energy. This suggests designing the cut stabilization at the level of these energy components, rather than transferring the mass--stiffness stabilization of continuous CutFEM directly to the EDG spatial operator.

Following this idea, we develop a high-order cut energy-based discontinuous Galerkin (CutEDG) method for the second-order acoustic wave equation on complex domains and across stationary material interfaces. Ghost-penalty terms are introduced into the two local bilinear forms associated with the discrete energy in the cut region, while no additional ghost penalty is added to the spatial coupling and numerical-flux terms of the original EDG formulation. The stabilization therefore restores control of the degrees of freedom on arbitrarily small cut cells without altering the underlying EDG flux structure. This construction preserves the energy mechanism that motivates the EDG formulation. More precisely, if $E_h(t)$ denotes the standard EDG energy on the physical domain and $D_h(t)\geq 0$ denotes the contribution of the cut stabilization, the semidiscrete CutEDG method satisfies an estimate of the form $\frac{\mathrm d}{\mathrm dt}
    \bigl(E_h(t)+D_h(t)\bigr)
    \leq 0$ (see \eqref{eq:discrete_energy_stable}). The additional term $D_h$ restores the control that is lost on arbitrarily small physical intersections, so that the augmented energy remains uniformly  meaningful with respect to the cut configuration. In this sense, the ghost stabilization extends the original EDG energy to the cut region while retaining its stability mechanism.

In this work, we establish stability and a priori error estimates for the proposed method and show that the designed high-order accuracy is retained. The formulation is further extended to stationary interface problems, where the interface may cut arbitrarily through the Cartesian background mesh. Numerical experiments demonstrate robustness with respect to decreasing cut-cell fractions and show that the explicit stability restriction is governed by the background mesh scale rather than by the smallest physical cut. We also numerically investigate the conditioning of the stabilized mass matrix. 
At the polynomial orders considered, the corresponding condition numbers remain bounded with respect to the cut position and are of the same order of magnitude as, while slightly smaller than, those observed for the analogous high-order continuous CutFEM discretization \cite{sticko2019higher}, indicating favorable algebraic behavior for high-order explicit wave propagation.

The main contributions are therefore threefold. First, we develop a high-order unfitted EDG formulation for both single-domain and stationary interface wave problems on Cartesian meshes. Second, we construct a ghost stabilization tailored to the local energy form of EDG, restoring uniform control on arbitrarily small cut cells while leaving the original spatial coupling and numerical-flux structure unchanged. Third, we establish stability and high-order error estimates and demonstrate numerically cut-independent explicit stability and favorable conditioning of the stabilized discrete operators.

The remainder of this paper is organized as follows. Section \ref{sec:CutEDG method} first reviews the EDG method and develops the proposed CutEDG formulation for acoustic wave equations. Stability and a priori error estimates are established on Cartesian meshes in Section \ref{sec:stability}. Section \ref{sec:interface_problem} considers stationary interface problems and presents the corresponding CutEDG formulation. Section \ref{sec:num} provides a comprehensive set of numerical experiments to verify the theoretical convergence rates and demonstrate the accuracy, stability, and robustness of the proposed method. Finally, concluding remarks are given in Section \ref{sec:concluding remarks}.

\section{The CutEDG method}\label{sec:CutEDG method}

In this section, we develop a CutEDG method for the second-order wave equation. We first introduce the model problem and the background mesh, and then briefly review the EDG formulation of \cite{appelo2015new}. To handle the difficulties arising from arbitrarily small cut cells on unfitted meshes, we augment the EDG formulation with a ghost-penalty stabilization.

\subsection{Model problem and fitted EDG formulation}

We consider the second-order wave equation
\begin{subequations}\label{eq:wave}
\begin{align}
u_{tt} &= \Delta u + f(\mathbf{x},t),
&& \mathbf{x}\in\Omega,\quad t\in(0,T), \label{eq:wave_pde}\\
u &= 0,
&& \mathbf{x}\in\Gamma,\quad t\in[0,T], \label{eq:wave_bc}\\
u(\mathbf{x},0) &= u_0(\mathbf{x}),\qquad
u_t(\mathbf{x},0)=v_0(\mathbf{x}),
&& \mathbf{x}\in\Omega. \label{eq:wave_ic}
\end{align}
\end{subequations}
Here, $\Omega\subset\mathbb{R}^d$ is a bounded physical domain with boundary $\Gamma=\partial\Omega$. The function $f$ denotes the source term, and $u_0$ and $v_0$ are the initial displacement and velocity, respectively. 

Introducing the velocity variable $v=u_t$, the wave equation can be written as the first-order system in time
\begin{equation}\label{eq:reformulation}
\begin{aligned}
u_t &= v,\\
v_t &= \Delta u + f.
\end{aligned}
\end{equation}
We first recall the EDG formulation on a boundary-fitted mesh (see, e.g., \cite{appelo2015new}). This formulation serves as the reference discretization from which the cut-cell method will be constructed. Let $\mathcal{T}_h^{\mathrm{fit}}$ be a boundary-fitted tessellation of $\Omega$ into non-overlapping elements $K$ such that $\overline{\Omega}=\bigcup_{K\in\mathcal{T}_h^{\mathrm{fit}}}\overline{K}$. 
Then define the element-wise polynomial space
\begin{equation}\label{eq:approximation_space}
V_{h,\mathrm{fit}}^{r}
=\left\{w_h: w_h\in\mathcal{P}^{r}(K) \quad \forall K\in\mathcal{T}_h^{\mathrm{fit}} \right\}.
\end{equation}
Here, $\mathcal{P}^{r}(K)$ denotes the polynomial approximation space of degree $r$ on $K$. 

Motivated by the continuous energy of the wave equation, the EDG scheme considers the discrete energy
\begin{equation}\label{eq:discrete_energy}
E_h(t)=\frac12\sum_{K\in\mathcal{T}_h^{\mathrm{fit}}}\int_K
\left( |v_h|^2+|\nabla u_h|^2\right) \,d\mathbf{x},
\end{equation}
and constructs the spatial discretization retains such discrete energy structure. Specifically, on each element $K\in\mathcal{T}_h^{\mathrm{fit}}$, one seeks $(u_h,v_h)\in V_{h,\mathrm{fit}}^{p}\times V_{h,\mathrm{fit}}^{q}$ such that, for all $(\phi_u,\phi_v,\widetilde \phi)\in V_{h,\text{fit}}^p(K) \times  V_{h,\text{fit}}^q(K) \times V_{h,\text{fit}}^0(K)$,
\begin{subequations}\label{eq:EDG}
\begin{align}
\int_K \big((u_h)_t-v_h\big)\widetilde{\phi} \,d\mathbf{x} &=0,\label{eq:EDG_3}\\
\int_K \nabla\big((u_h)_t-v_h\big)\cdot\nabla\phi_u\,d\mathbf{x} &= \int_{\partial K} (\widehat{v_h}-v_h) \nabla\phi_u\cdot\mathbf{n}\,dS,\label{eq:EDG_1} \\
\int_K (v_h)_t\phi_v + \nabla u_h\cdot\nabla\phi_v \,d\mathbf{x} &= \int_{\partial K} \widehat{\nabla u_h}\cdot\mathbf{n}\,\phi_v dS + \int_K f\phi_v\,d\mathbf{x}. \label{eq:EDG_2}
\end{align}
\end{subequations}
Here, $\mathbf{n}$ denotes the outward unit normal to $\partial K$. To define the numerical fluxes, consider an interior face $F$ shared by two neighboring elements $K^+$ and $K^-$. For a scalar or vector-valued discrete function $w$, we denote by $w^\pm$ its traces on $F$ taken from the interior of $K^\pm$, respectively. Let $\mathbf{n}^\pm$ be the outward unit normals to $K^\pm$ on $F$, so that $\mathbf{n}^-=-\mathbf{n}^+$. For a scalar quantity $w$ and a vector quantity $\mathbf q$, we define the corresponding normal jumps by
\begin{equation}\label{eq:jump_definition}
[\![w]\!]=w^+\mathbf{n}^+ + w^-\mathbf{n}^-, \qquad
[\![\mathbf{q}]\!] =  \mathbf{q}^+\cdot\mathbf{n}^+ + \mathbf{q}^-\cdot\mathbf{n}^-.
\end{equation}
In particular,
\[[\![\nabla u_h]\!]=\nabla u_h^+\cdot\mathbf{n}^+ + \nabla u_h^-\cdot\mathbf{n}^-.\]
The numerical fluxes on an interior face are taken as
\begin{equation}\label{eq:EDG_fluxes}
\begin{aligned}
\widehat{v_h} &= \alpha v_h^+ + (1-\alpha)v_h^- - \tau[\![\nabla u_h]\!], \\
\widehat{\nabla u_h} &= (1-\alpha)\nabla u_h^+ + \alpha\nabla u_h^- - \beta[\![v_h]\!].
\end{aligned}
\end{equation}
The parameters $\alpha$, $\tau$, and $\beta$ determine the particular choice of numerical flux. On the physical boundary $\Gamma$, we denote by $u_h^{in}$ and $v_h^{in}$ the interior traces of $u_h$ and $v_h$, the boundary numerical fluxes are defined consistently with the prescribed homogeneous Dirichlet data, namely, 
\begin{equation}\label{eq:flux_physical_bdry}
\widehat{v_h} = 0, \qquad \widehat{\nabla u_h} = \nabla u_h^{in}.
\end{equation}
For homogeneous Neumann boundary conditions, we impose
\begin{equation}\label{eq:flux_physical_bdry_Neumann}
\widehat{v_h} = v_h^{in}, \qquad \widehat{\nabla u_h}\cdot\mathbf{n}  = 0.
\end{equation}

We note that \eqref{eq:EDG_1} controls only the gradient of $(u_h)_t-v_h$ and therefore leaves its elementwise constant component undetermined. The additional relation \eqref{eq:EDG_3} determines this constant mode by enforcing the corresponding elementwise mean condition. Thus, \eqref{eq:EDG_1} and \eqref{eq:EDG_3} together provide a complete discrete representation of the relation $u_t=v$.

The formulation \eqref{eq:EDG} assumes that the computational mesh conforms to the physical boundary. We next remove this restriction by embedding $\Omega$ into a Cartesian background mesh and restricting the physical integrals to the portions of the background cells lying inside the physical domain $\Omega$.

\subsection{Background mesh and cut-cell geometry}

Let $\widetilde{\mathcal T}_h$ be a Cartesian background mesh covering the physical domain $\Omega$. The mesh is constructed independently of the physical boundary $\Gamma=\partial\Omega$, so that $\Gamma$ may intersect the background cells without conforming to their faces. We denote a generic background cell by $K$ and define
\[h_K=\operatorname{diam}(K), \qquad h=\max_{K\in\widetilde{\mathcal T}_h}h_K.\]

The discretization is supported only on those background cells having a nonzero intersection with the physical domain. This defines the active mesh and its associated computational domain,
\begin{equation}\label{eq:active_mesh}
\mathcal T_h = \left\{ K\in\widetilde{\mathcal T}_h:|K\cap\Omega|>0 \right\},\qquad \Omega_h = \bigcup_{K\in\mathcal T_h}K.
\end{equation}
For each $K\in\mathcal T_h$, we denote its physical part by
\begin{equation}\label{eq:physical_cell}
K_\Omega:=K\cap\Omega.
\end{equation}
Thus, although the discrete functions are represented on the full background cell $K$, the physical-domain contributions are evaluated only over $K_\Omega$. The active mesh $\mathcal T_h$ is naturally decomposed according to the interaction of its cells with the physical boundary. We define
\begin{equation}\label{eq:interior_cut_cells}
\mathcal T_h^{\mathrm{int}} = \left\{ K\in\mathcal T_h:
\overline K\subset\Omega\right\},
\qquad \mathcal T_h^\Gamma =  \mathcal T_h\setminus\mathcal T_h^{\mathrm{int}}.
\end{equation}
Cells in $\mathcal T_h^{\mathrm{int}}$ are unaffected by the geometric cut and satisfy $K_\Omega=K$, whereas cells in $\mathcal T_h^\Gamma$ are intersected by the physical boundary and generally satisfy $K_\Omega\subsetneq K$. See Figure \ref{fig:Illustration_2d} for an illustration in 2D.

\begin{figure}[htbp]
\centering
\begin{tikzpicture}[scale=0.9]
    \def\nx{8}
    \def\ny{6}
    \pgfmathsetmacro{\dx}{6/\nx}
    \pgfmathsetmacro{\dy}{4/\ny}
    \foreach \i in {0,...,\nx} {
        \draw[gray!50, dashed, thin] (\i*\dx,0) -- (\i*\dx,\ny*\dy);
    }
    \foreach \j in {0,...,\ny} {
        \draw[gray!50, dashed, thin] (0,\j*\dy) -- (\nx*\dx,\j*\dy);
    }
    \def\cx{3.0}
    \def\cy{2.0}
    \def\R{1.8}
    \fill[gray!5, opacity=0.2] (0,0) rectangle (\nx*\dx,\ny*\dy);
    \fill[blue!10, opacity=0.3, even odd rule] (\cx,\cy) circle (\R);
    \foreach \i/\j in {2/2,2/3,3/1,3/2,3/3,3/4,4/1,4/2,4/3,4/4,5/2,5/3} {
        \pgfmathsetmacro{\x}{\i*\dx}
        \pgfmathsetmacro{\y}{\j*\dy}
        \fill[blue!10, opacity=0.3] (\x,\y) rectangle (\x+\dx,\y+\dy);
    }
    \foreach \i/\j in {2/2,2/3,3/1,3/2,3/3,3/4,4/1,4/2,4/3,4/4,5/2,5/3} {
        \pgfmathsetmacro{\x}{\i*\dx}
        \pgfmathsetmacro{\y}{\j*\dy}
        \fill[pattern=north east lines, pattern color=blue!30] (\x,\y) rectangle (\x+\dx,\y+\dy);
    }
    \foreach \i/\j in {0/0,0/1,0/2,0/3,0/4,0/5,1/0,1/5,6/0,6/5,7/0,7/1,7/2,7/3,7/4,7/5} {
        \pgfmathsetmacro{\x}{\i*\dx}
        \pgfmathsetmacro{\y}{\j*\dy}
        \fill[pattern=north east lines, pattern color=gray!30] (\x,\y) rectangle (\x+\dx,\y+\dy);
    }

    \foreach \i/\j in {1/1,1/2,1/3,1/4,2/0,2/1,2/4,2/5,3/0,3/5,4/0,4/5,5/0,5/1,5/4,5/5,6/1,6/2,6/3,6/4} {
        \pgfmathsetmacro{\x}{\i*\dx}
        \pgfmathsetmacro{\y}{\j*\dy}
        \fill[pattern=north east lines, pattern color=green!40] (\x,\y) rectangle (\x+\dx,\y+\dy);
        \draw[thick, green!60] (\x,\y) rectangle (\x+\dx,\y+\dy);
    }
    \draw[thick, red] (\cx,\cy) circle (\R) node[above right, xshift=0.8cm, yshift=0.7cm, black] {$\Gamma$};

    \node at (\cx-0.2,\cy-0.2) {$\Omega$};   
    \node at (5.0,3.5) {$\widetilde{\mathcal{T}}_h$};     
    \draw[->] (-0.5,0) -- (6.8,0) node[right] {$x$};
    \draw[->] (0,-0.5) -- (0,4.8) node[above] {$y$};
    
\end{tikzpicture}
\caption{Illustration of an unfitted two-dimensional discretization.
The physical domain $\Omega$ is bounded by $\Gamma$ and embedded in the
Cartesian background mesh $\widetilde{\mathcal{T}}_h$. The blue cells
are the interior cells fully contained in $\Omega$, whereas the green
cells are the cut cells intersected by $\Gamma$, denoted as $\mathcal{T}_h^{\mathrm{int}}$ and $\mathcal{T}^{\Gamma}_h$ respectively. The union of the
interior and cut cells forms the computational mesh $\mathcal{T}_h$ used for the
discretization of $\Omega$. }
\label{fig:Illustration_2d}
\end{figure}

To distinguish the interfaces associated with the background mesh from those induced by the physical boundary, let
\begin{equation}\label{eq:background_interior_faces}
\mathcal F_h^{\mathrm{int}} = \left\{ F=\partial K^+\cap\partial K^-: K^\pm\in\mathcal T_h\,; \,\operatorname{meas}_{d-1}(F)>0 \right\}
\end{equation}
be the set of interior faces of the active background mesh. Only the parts of these faces lying inside $\Omega$ enter the physical EDG fluxes; we therefore introduce
\begin{equation}\label{eq:physical_interior_faces}
\mathcal F_h^\Omega = \left\{ F\cap\Omega: F\in\mathcal F_h^{\mathrm{int}}\, ; \, \operatorname{meas}_{d-1}(F\cap\Omega)>0 \right\}.
\end{equation}
The portions of the physical boundary contained in the cut cells are collected in
\begin{equation}\label{eq:physical_boundary_segments}
\mathcal F_h^\Gamma =  \left\{ K\cap\Gamma:
K\in\mathcal T_h^\Gamma\,; \, \operatorname{meas}_{d-1}(K\cap\Gamma)>0
\right\}.
\end{equation}
Hence, three geometric objects play distinct roles in the CutEDG formulation: $K_\Omega$ corresponds to the domain of the volume integrals, $\mathcal F_h^\Omega$ contains faces for the numerical fluxes between neighboring active cells, and $\mathcal F_h^\Gamma$ represents elementwise physical boundary on which the boundary condition is imposed. This distinction will also be useful when introducing the ghost-penalty stabilization, which is defined instead on selected \emph{full} background faces.

\begin{figure}[htbp]
\centering
\begin{tikzpicture}[
    scale=1.15,
    line cap=round,
    line join=round,
    cell/.style={
        gray!60,
        dashed,
        line width=0.9pt
    },
    curve/.style={
        red!80!black,
        line width=1.25pt
    }
]
    \def\s{2.4}

    % =====================================================
    % (a) The curve intersects the right edge twice
    % =====================================================
    \draw[cell] (0,0) rectangle (\s,\s);

    % (a) Smooth interface crossing the right edge twice
    \draw[curve]
        (0.75,1.98)
        .. controls (1.45,1.90) and (2.75,1.75) ..
        (2.78,1.48)
        .. controls (2.80,1.20) and (1.55,0.78) ..
        (0.80,0.62);

    % Black interface label
    \node[black] at (1.85,2.13) {$\Gamma$};
    \node at (1.20,-0.38) {(a)};

    % =====================================================
    % (b) A strongly curved interface inside the cell
    % =====================================================
    \begin{scope}[xshift=4.2cm]
        \draw[cell] (0,0) rectangle (\s,\s);

        % The curve enters from the left, bends inside
        % the cell, and exits through the bottom edge.
        \draw[curve]
            (0,1.68)
            .. controls (0.52,1.78) and (1.15,1.78) ..
            (1.48,1.48)
            .. controls (1.82,1.18) and (1.48,0.98) ..
            (1.05,0.76)
            .. controls (0.67,0.56) and (0.75,0.34) ..
            (1.08,0.00);

        % Black interface label
        \node[black] at (1.60,2.03) {$\Gamma$};
        \node at (1.20,-0.38) {(b)};
    \end{scope}

\end{tikzpicture}
\caption{Illustrations of two interface configurations excluded by the geometric assumptions. (a) The interface $\Gamma$ intersects the same element edge at two distinct points. (b) 
The interface exhibits significant curvature within an element, such that the straight line segment connecting its intersection points with $\partial K$ does not provide an adequate approximation of the interface geometry.}
\label{fig:assumptions}
\end{figure}

Finally, we assume that the background mesh is sufficiently fine to satisfy certain geometric conditions that simplify numerical integration over cut elements. Specifically, we require that the interface $\Gamma$ intersects each mesh face at most once and that, within each cut element, the interface can be parametrized by the straight line segment connecting its two intersection points with the element boundary (Figure \ref{fig:assumptions} illustrates two configurations that violate these assumptions and are therefore excluded from the present study). These assumptions are generally reasonable for sufficiently smooth interfaces on adequately refined meshes. Without these conditions, resolving the cut-cell geometry and performing accurate numerical integration would require more elaborate and computationally expensive geometric reconstruction and quadrature procedures.

\subsection{CutEDG discretization}

We now extend the fitted EDG formulation \eqref{eq:EDG} to the active background mesh $\mathcal{T}_h$. The discrete functions are represented by polynomials on the entire background cells, while the physical volume and boundary integrals are restricted to the intersections with $\Omega$. We define the approximation spaces on the active mesh by
\begin{equation}\label{eq:cut_finite_element_spaces}
\begin{aligned}
V_h^{r} &= \left\{ w_h: w_h\in\mathcal{P}^{r}(K) \quad \forall K\in\mathcal{T}_h \right\}.
\end{aligned}
\end{equation}
The unstabilized CutEDG formulation is obtained by replacing each fitted element $K$ in \eqref{eq:EDG} by its physical portion $K_\Omega$ defined in \eqref{eq:physical_cell}. Thus, for every $K\in\mathcal{T}_h$, we seek $(u_h,v_h)\in V_h^p\times V_h^q$ such that, for all $(\phi_u,\phi_v,\widetilde \phi)\in  V_h^p \times  V_h^q \times V_h^{0}$,
\begin{subequations}\label{eq:cutEDG_unstabilized}
\begin{align}
\int_{K_\Omega} \big((u_h)_t-v_h\big)\widetilde{\phi} \,d\mathbf{x} &=0, \label{eq:cutEDG_mean} \\
\int_{K_\Omega}  \nabla\big((u_h)_t-v_h\big)\cdot\nabla\phi_u
\,d\mathbf{x} &= \int_{\partial K_\Omega} (\widehat{v_h}-v_h) \nabla\phi_u\cdot\mathbf{n}_{\Omega,K}\,dS, \label{eq:cutEDG_u} \\
\int_{K_\Omega} (v_h)_t\phi_v + \nabla u_h\cdot\nabla\phi_v  \,d\mathbf{x} &= \int_{\partial K_\Omega}  \widehat{\nabla u_h} \cdot\mathbf{n}_{\Omega,K}\,\phi_v \,dS + \int_{K_\Omega} f\phi_v\,d\mathbf{x}, \label{eq:cutEDG_v}.
\end{align}
\end{subequations}
Here, $\mathbf{n}_{\Omega,K}$ denotes the outward unit normal to the physical cell $K_\Omega$. On a physical interior face belonging to $\mathcal{F}_h^\Omega$, the numerical fluxes are inherited directly from the fitted EDG formulation. On a boundary segment belonging to $\mathcal{F}_h^\Gamma$, the numerical fluxes are replaced by the corresponding boundary fluxes used to impose the prescribed homogeneous Dirichlet condition.

For an interior cell $K\in\mathcal{T}_h^{\mathrm{int}}$, one has $K_\Omega=K$. Consequently, the local CutEDG equations reduce exactly to the fitted EDG equations. The distinction arises only in the vicinity of the physical boundary, where the measure of $K_\Omega$ may be arbitrarily small relative to that of the background cell $K$.

\subsection{Small-cut-cell problem and ghost-penalty stabilization}

To characterize the position of the physical boundary within an active cell, we define the volume fraction
\begin{equation}\label{eq:cut_fraction}
\eta_K = \frac{|K_\Omega|}{|K|}, \qquad K\in\mathcal{T}_h.
\end{equation}
For an interior cell, $\eta_K=1$, whereas for a cut cell $\eta_K\in(0,1)$ and may become arbitrarily small as the location of $\Gamma$ changes relative to the background grid. When $\eta_K \ll 1$, the physical volume contributions associated with $K$ become small, even though the discrete polynomial on $K$ is still represented by degrees of freedom defined over the full background cell. As a consequence, the discrete system \eqref{eq:cutEDG_unstabilized} may lose uniform control of modes supported predominantly in the fictitious portion $K\setminus\Omega$. This is the familiar small-cut-cell problem and may lead to severe deterioration of the conditioning of the resulting algebraic system.

To restore control in the cut region, motivated by \cite{sticko2019higher}, we introduce a ghost-penalty stabilization on background faces adjacent to cut cells. Define
\begin{equation}\label{eq:ghost_faces}
\mathcal{F}_h^g = \left\{ F=\partial K^+\cap\partial K^- \in\mathcal{F}_h^{\mathrm{int}}: K^+\in\mathcal{T}_h^\Gamma
\text{ or } K^-\in\mathcal{T}_h^\Gamma \right\}.
\end{equation}
An important distinction is that the stabilization is integrated over the full background face $F$, rather than only over its physical portion $\Omega_F = F\cap\Omega$. 
The ghost penalty therefore transfers control from the physical domain into the neighboring fictitious parts of the active mesh. Precisely, for a polynomial degree $r$, we define
\begin{equation}\label{eq:ghost_penalty}
J_r(w,z) = \sum_{F\in\mathcal{F}_h^g} \sum_{\ell=0}^{r} \omega_\ell
\frac{h^{2\ell+1}} {(2\ell+1)(\ell!)^2} \int_F [\![\partial_{n_F}^{\ell}w]\!]\, [\![\partial_{n_F}^{\ell}z]\!] \,dS,
\end{equation}
where $n_F$ is a fixed unit normal associated with $F$, $\partial_{n_F}^{\ell}$ denotes the $\ell$th-order directional derivative in the $n_F$ direction, and $\omega_\ell>0$ is a stabilization weight. Since the approximation spaces are discontinuous, the contribution with $\ell=0$ is retained and penalizes jumps of the discrete functions themselves.

Because $J_r(\cdot,\cdot)$ is defined as a sum over ghost-penalty stabilization  background faces, it is most naturally interpreted as a global stabilization form. Accordingly, the stabilized CutEDG formulation is obtained by adding the appropriate ghost-penalty contributions to the assembled CutEDG equations. In abstract form, the two evolution equations may be written as 
\begin{equation}\label{eq:cutEDG_stabilized_1}
\begin{aligned}
&\sum_{K\in\mathcal{T}_h} \int_{K_\Omega} 
\nabla\big((u_h)_t-v_h\big)\cdot\nabla\phi_u \,d\mathbf{x} + \gamma_M^{u}h^{-2} J_p\big((u_h)_t,\phi_u\big) \\
= &\sum_{K\in\mathcal{T}_h} \int_{\partial K_\Omega} (\widehat{v_h}-v_h) \nabla\phi_u\cdot\mathbf{n}_{\Omega,K} \,dS,
\end{aligned}
\end{equation}
and
\begin{equation}\label{eq:cutEDG_stabilized_2}
\begin{aligned}
&\sum_{K\in\mathcal{T}_h} \int_{K_\Omega}  \left( (v_h)_t\phi_v + \nabla u_h\cdot\nabla\phi_v \right) \,d\mathbf{x} + \gamma_M^{v} J_q\big((v_h)_t,\phi_v\big) \\
= &\sum_{K\in\mathcal{T}_h} \int_{\partial K_\Omega} \widehat{\nabla u_h}\cdot\mathbf{n}_{\Omega,K}\,\phi_v \,dS + \sum_{K\in\mathcal{T}_h} \int_{K_\Omega} f\phi_v\,d\mathbf{x}.
\end{aligned}
\end{equation}
The element-wise mean condition associated with the relation $u_t=v$ is imposed in the stabilized formulation as 
\begin{equation}\label{eq:cutEDG_stabilized_3}
\sum_{K\in \mathcal T_h} \int_{K_\Omega} h^{-2} \big((u_h)_t - v_h\big)\widetilde \phi\,d{\bf x} + \gamma_M^{u}h^{-2} J_p\big((u_h)_t,\widetilde\phi\big) = 0.
\end{equation}
We note that the global form \eqref{eq:cutEDG_stabilized_1}--\eqref{eq:cutEDG_stabilized_3} makes explicit that each ghost-penalty stabilization face contributes only once to the assembled system. Nevertheless, the resulting couplings remain local, and only cut cells and the elements sharing a stabilized face with them are coupled beyond the fitted meshes: the physical volume terms involve individual physical cells $K_\Omega$, the numerical fluxes couple neighboring elements across the relevant portions of $\partial K_\Omega$, and the ghost-penalty terms couple only pairs of elements sharing a face in $\mathcal F_h^g$. Consequently, cells away from the cut region receive no ghost-penalty contribution, and their discretization coincides with the fitted EDG scheme.

\begin{remark}
The gradient equation does not control the elementwise constant component of $(u_h)_t-v_h$. The mean-value condition is therefore a complementary algebraic constraint required to determine this otherwise undetermined mode. To incorporate this constraint consistently into the stabilized formulation, the mean equation \eqref{eq:cutEDG_mean} is first scaled by $h^{-2}$. A ghost-penalty term acting on $(u_h)_t$ is then added with the same $h^{-2}$ scaling as in \eqref{eq:cutEDG_stabilized_1}. This scaling places the mean constraint on the same dimensional level as the gradient equation and provides additional control of the constant mode in the cut region (see \eqref{eq:cutEDG_stabilized_3}). Nonetheless, since $\widetilde \phi$ is piecewise constant, only the zeroth-order contribution of $J_p\big((u_h)_t,\widetilde\phi\big)$ is nonzero.
\end{remark}

\begin{remark}
The ghost penalty introduces additional coupling only across faces in $\mathcal F_h^g$. Hence, the coupling induced by the cut-cell stabilization is confined to a neighborhood of the physical boundary. Away from this region, the discrete operator retains the same element-local structure as the fitted EDG method. In particular, the ghost penalty modifies only the part of the discrete system associated with cut cells and their neighboring elements. Unlike standard CutFEM formulations, in which ghost-penalty stabilization is typically applied to both the mass and stiffness matrices, the proposed CutEDG method incorporates the ghost penalty only into the mass matrix.
\end{remark}

\begin{remark}
The CutEDG formulation \eqref{eq:cutEDG_unstabilized} inherits its underlying structure from the EDG discretization \eqref{eq:EDG}. In particular, the physical-domain terms retain the same energy structure as in the fitted formulation. The ghost-penalty terms introduce additional stabilization in the cut region, and their contribution to the discrete energy balance depends on the particular stabilization operators and numerical fluxes. The resulting energy properties of the stabilized CutEDG formulation are therefore analyzed separately below.
\end{remark}

\section{Stability and error analysis}
\label{sec:stability}

In this section, we analyze the stability and convergence of the proposed CutEDG method \eqref{eq:cutEDG_stabilized_1}--\eqref{eq:cutEDG_stabilized_3}. We first establish a semi-discrete energy estimate and then derive an a priori error bound.

\subsection{Discrete energy stability}

We begin with the homogeneous problem, for which the source term and Dirichlet boundary data vanish. The main stability result is stated below.

\begin{theorem}
\label{thm:semi_discrete_stability}
Assume that $f=0$, and let $\tau,\beta\geq0$ in \eqref{eq:EDG_fluxes} . Then the semi-discrete CutEDG solution obtained from \eqref{eq:cutEDG_stabilized_1}--\eqref{eq:cutEDG_stabilized_3} satisfies
\begin{equation}\label{eq:discrete_energy_stable}
\frac{d}{dt}\left(E_h(t) + \gamma_M^{(u)}h^{-2}J_p(u_h,u_h) +\gamma_M^{(v)}J_q(v_h,v_h)\right)\leq0,
\end{equation}
where the discrete energy is defined by
\begin{equation} \label{eq:discrete_energy}
E_h(t) = \frac{1}{2} \sum_{K\in\mathcal{T}_h} \int_{K_\Omega}|\nabla u_h|^2+|v_h|^2\,d\mathbf{x}.
\end{equation}
\end{theorem}

\begin{proof}
By the construction \eqref{eq:ghost_penalty}, the ghost-penalty forms are symmetric and nonnegative,
\[J_p(w,w)\geq0, \qquad J_q(w,w)\geq0.\]
It follows from the definition of $E_h(t)$ in (\ref{eq:discrete_energy}) and the
nonnegativity of the ghost-penalty forms that
\[
E_h(t)+\gamma_M^{(u)}h^{-2}J_p(u_h,u_h)
+\gamma_M^{(v)}J_q(v_h,v_h)\ge 0.
\]

For each $K\in\mathcal{T}_h$, we choose $\phi_u=u_h$ and $\phi_v=v_h$ in the stabilized global CutEDG formulation \eqref{eq:cutEDG_stabilized_1}--\eqref{eq:cutEDG_stabilized_2}, and sum the resulting equations. By symmetry of the ghost-penalty forms,
\begin{equation}
J_p((u_h)_t,u_h) = \frac{1}{2}\frac{d}{dt}J_p(u_h,u_h),\quad 
J_q((v_h)_t,v_h) = \frac{1}{2}\frac{d}{dt}J_q(v_h,v_h).
\end{equation}
The volume coupling terms cancel after summation, giving
\begin{equation}
\begin{aligned}
&\frac{d}{dt}\left(E_h(t)+\gamma_M^{(u)}h^{-2}J_p(u_h,u_h) +\gamma_M^{(v)}J_q(v_h,v_h)\right) \\
= &\sum_{K\in\mathcal{T}_h} \int_{\partial K_\Omega} (\widehat{v_h}-v_h) \nabla u_h\cdot\mathbf{n} + \widehat{\nabla u_h}\cdot\mathbf{n} v_h \,dS.
\end{aligned}
\end{equation}
Using the numerical fluxes defined in \eqref{eq:EDG_fluxes} and combining the contributions from the two elements sharing each physical interior face yields
\begin{equation} \label{eq:energy_dissipation}
\begin{aligned}
&\frac{d}{dt}\left(E_h(t) + \gamma_M^{(u)}h^{-2}J_p(u_h,u_h) +\gamma_M^{(v)}J_q(v_h,v_h)\right) \\
= &- \sum_{F\in\mathcal{F}_h^\Omega} \int_F
\tau [\![\nabla u_h]\!]^2\,dS + \beta |[\![ v_h]\!]|^2\,dS.
\end{aligned}
\end{equation}
Since $\tau,\beta\geq0$, we obtain \eqref{eq:discrete_energy_stable}, which completes the proof.
\end{proof}

\subsection{A priori error estimate}

To establish the a priori error estimates for the CutEDG method, we introduce a Ritz projection for the displacement \(u\) and an \(L^2\)-projection for the velocity \(v\). These projections serve as auxiliary tools in the error analysis. In particular, we define the projections
$(\tilde u_h,\tilde v_h)\in V_h^{p}\times V_h^{q}$ by
\begin{equation}\label{eq:projection}
\int_{K_\Omega} \nabla(u-\tilde u_h)\cdot\nabla\phi_u \,d\mathbf{x} =0,\quad \int_{K_\Omega} (v - \tilde v_h) \phi_v\ d{\bf x} = 0,\quad 
\int_{K_\Omega} u-\tilde u_h \,d\mathbf{x} =0,
\end{equation}
for all $(\phi_u,\phi_v)\in V_h^{p}\times V_h^{q}$.  The initial discrete data are constructed using stabilized projection problems to account for the effect of small cut cells. Specifically, we set
\[
u_h(0)=\widetilde{u}_{h,0},
\qquad
v_h(0)=\widetilde{v}_{h,0},
\]
where \(\widetilde{u}_{h,0}\in V_h^p\) and \(\widetilde{v}_{h,0}\in V_h^q\) are obtained by solving
\begin{subequations}\label{eq:initial_projection}
\begin{align}
&\sum_{K\in\mathcal{T}_h}\int_{K_\Omega}
h^{-2}\widetilde{u}_{h,0}\widetilde{\phi}\,\mathrm{d}\mathbf{x}
+\gamma_M^{(u)}h^{-2}J_p(\widetilde{u}_{h,0},\widetilde{\phi})
=\sum_{K\in\mathcal{T}_h}\int_{K_\Omega}
h^{-2}u_0\widetilde{\phi}\,\mathrm{d}\mathbf{x},
\label{eq:initial_projection_a}\\
&\sum_{K\in\mathcal{T}_h}\int_{K_\Omega}
\nabla\widetilde{u}_{h,0}\cdot\nabla\phi_u\,\mathrm{d}\mathbf{x}
+\gamma_M^{(u)}h^{-2}J_p(\widetilde{u}_{h,0},\phi_u)
=\sum_{K\in\mathcal{T}_h}\int_{K_\Omega}
\nabla u_0\cdot\nabla\phi_u\,\mathrm{d}\mathbf{x},
\label{eq:initial_projection_b}\\
&\sum_{K\in\mathcal{T}_h}\int_{K_\Omega}
\widetilde{v}_{h,0}\phi_v\,\mathrm{d}\mathbf{x}
+\gamma_M^{(v)}J_q(\widetilde{v}_{h,0},\phi_v)
=\sum_{K\in\mathcal{T}_h}\int_{K_\Omega}
v_0\phi_v\,\mathrm{d}\mathbf{x}.
\label{eq:initial_projection_c}
\end{align}
\end{subequations}
for all $(\phi_u,\phi_v,\widetilde \phi)\in  V_h^p \times  V_h^q \times V_h^{0}$, respectively.

Despite the use of unfitted meshes, the projections employed in the error analysis retain the similar approximation properties under the assumed regularity of the exact solution; see, e.g., \cite{sticko2016stabilized, fu2021high, Gurkan2020}. 
We next derive an a priori error estimate for the semi-discrete CutEDG approximation \eqref{eq:cutEDG_stabilized_1}--\eqref{eq:cutEDG_stabilized_3}. 

\begin{theorem}\label{thm:error_estimate}
Consider the problem \eqref{eq:wave}. Assume that
\[u,u_t\in H^{p+1}(\Omega),\qquad v,v_t\in H^{q+1}(\Omega),\]
with the corresponding norms uniformly bounded for $t\in[0,T]$. Suppose further that
\begin{equation}\label{eq:degree}
p-2\leq q\leq p
\end{equation}
and that the initial discrete solution is chosen using the projections \eqref{eq:initial_projection}. 
Then
\begin{equation}
\label{eq:error_estimate}
\begin{aligned}
\|\nabla(u-u_h)\|_{L^2(\Omega)}^2 + \|v-v_h\|_{L^2(\Omega)}^2 \lesssim\;& h^{2\gamma} \left( |u|_{H^{p+1}(\Omega)}^2 + |v|_{H^{q+1}(\Omega)}^2 \right)\\
&+h^{2p}|u_t|_{H^{p+1}(\Omega)}^2+h^{2(q+1)}|v_t|_{H^{q+1}(\Omega)}^2,
\end{aligned}
\end{equation}
where $\gamma=\min\{p',q'\}$, with
\begin{equation}
\label{eq:pq_prime}
p'=\begin{cases}
p-1, & \tau=0,\\
p-\dfrac{1}{2}, & \tau>0,
\end{cases} \qquad q' = \begin{cases} q, & \beta=0,\\
q+\dfrac{1}{2}, & \beta>0.
\end{cases}
\end{equation}
Here and below, $A\lesssim B$ means that
$A\leq CB$ for a constant $C>0$ independent of $h$.
\end{theorem}

\begin{proof}
Let $\mathcal U=(u,v)$, $\mathcal U_h=(u_h,v_h)$, and $\Phi=(\phi_u,\phi_v)\in V_h^p\times V_h^q$. For convenience, we introduce the global residual form associated with the two energy equations of the CutEDG scheme,
\begin{equation}\label{eq:error_bilinear_form}
\begin{aligned}
B_h(\mathcal U_h,\Phi) := &\sum_{K\in\mathcal{T}_h} \int_{K_\Omega}
\nabla\big((u_h)_t-v_h\big)\cdot\nabla\phi_u + (v_h)_t\phi_v + \nabla u_h\cdot\nabla\phi_v \,d\mathbf{x} \\
&- \sum_{K\in\mathcal{T}_h} \int_{\partial(K\cap\Omega)} (\widehat{v_h}-v_h) (\nabla\phi_u\cdot\mathbf{n}) + (\widehat{\nabla u_h}\cdot\mathbf{n})
\phi_v \,dS\\
&+
\gamma_M^{(u)}h^{-2}
J_p((u_h)_t,\phi_u)
+
\gamma_M^{(v)}
J_q((v_h)_t,\phi_v).
\end{aligned}
\end{equation}
The elementwise mean constraint for $(u_h)_t-v_h$ does not enter the energy argument and is therefore omitted from $B_h$. By consistency of the numerical fluxes and the ghost-penalty terms, the exact and numerical solutions satisfy
\[ B_h(\mathcal U,\Phi)=0, \qquad B_h(\mathcal U_h,\Phi)=0. \]
Hence the error $U-U_h$ satisfies
\begin{equation}\label{eq:error_equation}
B_h(\mathcal U-\mathcal U_h,\Phi)=0 \qquad \forall\,\Phi\in V_h^p\times V_h^q.
\end{equation}

To estimate the error $\mathcal U-\mathcal U_h$, we compare the CutEDG solution $\mathcal U_h$ with suitable projections \eqref{eq:projection} of the exact solution $\mathcal U$.
We now split the errors as
\begin{equation}
\label{eq:error_splitting}
\begin{aligned}
e_u := u-u_h &= \tilde e_u - \delta_u,  & \tilde e_u&:=\tilde u_h -u_h, & \delta_u&:= \tilde u_h - u,\\
e_v := v-v_h &= \tilde e_v - \delta_v,  & \tilde e_v &:=\tilde v_h -v_h,& \delta_v&:= \tilde v_h - v.
\end{aligned}
\end{equation}
Substituting \eqref{eq:error_splitting} into
\eqref{eq:error_equation} gives
\begin{equation}\label{eq:discrete_error_equation}
B_h((\tilde e_u,\tilde e_v),\Phi) = B_h((\delta_u,\delta_v),\Phi).
\end{equation}

Choose $\Phi = (\tilde e_u, \tilde e_v)$ and invoke \eqref{eq:error_bilinear_form}, we obtain
\begin{equation}\label{eq:simplified}
\begin{aligned}
&\dfrac{1}{2} \dfrac{d}{dt} \bigg(\sum_{K\in\mathcal T_h}\int_{K_\Omega} |\nabla\tilde{e}_u|^2 + \tilde{e}_v^2 \, d\mathbf{x} +\gamma_M^{(u)}h^{-2}J_p(\tilde{e}_u,\tilde{e}_u) + \gamma_M^{(v)}J_q(\tilde{e}_v,\tilde{e}_v)\bigg) \\
= & \underbrace{-\sum_{F\in\mathcal{F}_h^\Omega}\int_F [\![\nabla\tilde{e}_u]\!] \, \widehat{\delta_v} + [\![\tilde{e}_v]\!]\cdot\widehat{\nabla\delta_u} +  \beta [\![\tilde{e}_v]\!]^2+\tau [\![\nabla\tilde{e}_u]\!]^2 \, dS}_{\mathcal{R}_{\mathrm{flux}} } \\
     & + \gamma_M^{(u)}h^{-2}J_p\big((\delta_u)_t,\tilde{e}_u\big) 
     + \gamma_M^{(v)}J_q\big((\delta_v)_t,\tilde{e}_v\big),
\end{aligned}
\end{equation}
where we have used projection orthogonality to the volume integrals in $B_h((\delta_u,\delta_v),(\tilde e_u,\tilde e_v))$. To estimate the right-hand-side of \eqref{eq:simplified}, similar to \cite{appelo2015new}, the approximation properties of the projection defined in \eqref{eq:projection}, together with \eqref{eq:degree} and standard trace estimates yield
\begin{equation}\label{eq:projection_flux_estimate}
\mathcal{R}_{\mathrm{flux}} \lesssim \|\nabla \tilde e_u\|_{L^2(\Omega)}^2 + \|\tilde e_v\|_{L^2(\Omega)}^2 + h^{2p'}|u|_{H^{p+1}(\Omega)}^2 + h^{2q'}|v|_{H^{q+1}(\Omega)}^2,
\end{equation}
where $p'$ and $q'$ are given by \eqref{eq:pq_prime}. Additionally, for the stabilization terms in \eqref{eq:simplified}, we have
\begin{equation}\label{eq:ghost_projection_u}
    \begin{aligned}
\gamma_M^{(u)}h^{-2}J_p((\delta_u)_t,\tilde{e}_u)
        \lesssim & \ \gamma_M^{(u)}h^{-2}\sum_{F\in\mathcal{F}_h^g}\sum_{k=0}^pw_k\dfrac{h^{2k+1}}{(2k+1)(k!)^2}\int_F  |[\![\partial_{n_F}^k(\delta_u)_t]\!]|^2 + |[\![\partial_{n_F}^k \tilde{e}_u ]\!]|^2 \, dS  \\
        \lesssim &\  h^{-2} \sum_{F\in\mathcal{F}_h^g}\sum_{k=0}^pw_kh^{2k+1}h^{2(p+1)-2k-1}|u_t|^2_{H^{p+1}(\Omega)} + \gamma_M^{(u)}h^{-2}J_p(\tilde{e}_u,\tilde{e}_u)\\
        \lesssim &\  h^{2p}|u_t|^2_{H^{p+1}(\Omega)} + \gamma_M^{(u)}h^{-2} J_p(\tilde{e}_u,\tilde{e}_u).
    \end{aligned}
\end{equation}
and
\begin{equation}\label{eq:ghost_projection_v}
\begin{aligned}
\gamma_M^{(v)}J_q((\delta_v)_t,\tilde{e}_v)
          \lesssim&\  \gamma_M^{(v)}\sum_{F\in\mathcal{F}_h^g}\sum_{k=0}^q w_k\dfrac{h^{2k+1}}{(2k+1)(k!)^2}\int_F |[\![\partial_{n_F}^k(\delta_v)_t]\!]|^2 + |[\![\partial_{n_F}^k \tilde e_v]\!]|^2\, dS\\
         \lesssim & \ \sum_{F\in\mathcal{F}_h^g}\sum_{k=0}^q w_kh^{2k+1}h^{2(q+1)-2k-1}|v_t|^2_{H^{q+1}(\Omega)} + \gamma_M^{(v)}J_q(\tilde{e}_v,\tilde{e}_v) \\
        \lesssim & \ h^{2(q+1)}|v_t|^2_{H^{q+1}(\Omega)} + \gamma_M^{(v)}J_q(\tilde{e}_v,\tilde{e}_v).
\end{aligned}
\end{equation}
Plugging \eqref{eq:projection_flux_estimate}--\eqref{eq:ghost_projection_v} into \eqref{eq:simplified}, we get
\begin{equation}\label{eq:error_energy_inequality}
    \begin{aligned}
    &\dfrac{1}{2} \dfrac{d}{dt} \left(\sum_K\int_{K_\Omega}  |\nabla\tilde{e}_u|^2 + \tilde{e}_v^2\, d\mathbf{x} +\gamma_M^{(u)}h^{-2}J_p(\tilde{e}_u,\tilde{e}_u) + \gamma_M^{(v)}J_q(\tilde{e}_v,\tilde{e}_v)\right)  \\
    \lesssim & \ \sum_K\int_{K_\Omega} |\nabla\tilde{e}_u|^2+ \tilde{e}_v^2\, d\mathbf{x} + \gamma_M^{(u)}h^{-2}J_p(\tilde{e}_u,\tilde{e}_u) + \gamma_M^{(v)}J_q(\tilde{e}_v,\tilde{e}_v)\\
    &+ h^{2\gamma} \left( |u|^2_{H^{p+1}(\Omega)} + |v|^2_{H^{q+1}(\Omega)} \right) + h^{2p}|u_t|^2_{H^{p+1}(\Omega)} + h^{2(q+1)}|v_t|^2_{H^{q+1}(\Omega)}.
    \end{aligned}
\end{equation}
Since the discrete initial data are chosen from the corresponding projections, the discrete component of the initial error vanishes. Applying Gr\"onwall's inequality to \eqref{eq:error_energy_inequality} therefore yields
\begin{equation}
\begin{aligned}
& \|\nabla\tilde e_u\|_{L^2(\Omega)}^2 + \|\tilde e_v\|_{L^2(\Omega)}^2 + \gamma_M^{(u)}h^{-2}J_p(\tilde e_u,\tilde e_u) + \gamma_M^{(v)}J_q(\tilde e_v,\tilde e_v) \\
\lesssim &\  h^{2\gamma} \left( |u|_{H^{p+1}(\Omega)}^2 + |v|_{H^{q+1}(\Omega)}^2 \right) + h^{2p}|u_t|_{H^{p+1 (\Omega)}}^2 + h^{2(q+1)} |v_t|_{H^{q+1}(\Omega)}^2.
\end{aligned}
\end{equation}
Finally, combining this estimate with the approximation properties of the projections and applying the triangle inequality gives \eqref{eq:error_estimate}.
\end{proof}

\begin{remark}
The restriction \eqref{eq:degree} is imposed solely for the derivation of the error estimate; no corresponding relation between the polynomial degrees $p$ and $q$ is required for the stability of the scheme. In our numerical experiments, the choice $q=p-1$ yields particularly robust results, consistent with commonly used polynomial-degree combinations in standard EDG formulations (see, e.g., \cite{appelo2015new}). Moreover, for this choice of polynomial degrees, the suboptimal convergence rate stated in Theorem \ref{thm:error_estimate} can be improved to the optimal rate with special designed numerical fluxes by employing suitably constructed projection operators for which the relevant projection errors vanish on element boundaries. We refer to \cite{appelo2015new} for the corresponding analysis.
\end{remark}

\section{Interface problems}\label{sec:interface_problem}

We now extend the CutEDG formulation to wave propagation problems with discontinuous wave speeds across an internal material interface. Such transmission problems arise naturally in acoustic and seismic wave propagation through heterogeneous media. When the material interface has a complicated geometry, requiring the computational mesh to conform to the interface can lead to substantial mesh-generation costs. We therefore retain the unfitted framework developed in the preceding sections and employ a Cartesian background mesh whose construction is independent of the interface location.

\subsection{Model problem and interface coupling}\label{subsec:model_interface_problem}

Let the physical domain $\Omega$ be decomposed into two non-overlapping subdomains $\Omega_1$ and $\Omega_2$ such that $\overline{\Omega}=\overline{\Omega}_1\cup\overline{\Omega}_2$, $\Omega_1\cap\Omega_2=\emptyset$, with the internal interface $\Gamma_I=\partial\Omega_1\cap\partial\Omega_2$. For $i=1,2$, let $u_i$ denote the wave field in $\Omega_i$, and let $c_i>0$ denote the corresponding constant wave speed. In each subdomain, we consider
\begin{equation}\label{eq:interface_wave_equation}
(u_i)_{tt} = c_i^2\Delta u_i+f_i, \qquad (\mathbf{x},t)\in\Omega_i\times(0,T].
\end{equation}
The two subdomain solutions are coupled through the transmission conditions
\begin{equation}\label{eq:interface_transmission_conditions}
u_1=u_2, \qquad c_1^2\nabla u_1\cdot\mathbf{n}_1 + c_2^2\nabla u_2\cdot\mathbf{n}_2 = 0 \qquad \text{on }\Gamma_I,
\end{equation}
where $\mathbf{n}_i$ denotes the outward unit normal to $\Omega_i$ on $\Gamma_I$. The first condition enforces continuity of the wave field, whereas the second expresses conservation of the normal flux across the interface. The problem is supplemented with appropriate boundary conditions on $\partial\Omega$ and the initial conditions
\begin{equation}\label{eq:interface_initial_conditions}
u_i(\mathbf{x},0) = u_{0,i}(\mathbf{x}), \qquad (u_i)_t(\mathbf{x},0) = v_{0,i} (\mathbf{x}), \qquad \mathbf{x}\in\Omega_i, \qquad i=1,2.
\end{equation}
As in the single-domain case, we introduce $v_i=(u_i)_t$, so that \eqref{eq:interface_wave_equation} can be written as the first-order system in time
\begin{equation}\label{eq:interface_first_order_system}
\begin{cases}
(u_i)_t=v_i,\\
(v_i)_t=c_i^2\Delta u_i+f_i,
\end{cases} \qquad (\mathbf{x},t)\in\Omega_i\times(0,T], \qquad i=1,2.
\end{equation}
Since the interface is stationary, the transmission conditions imply
\begin{equation}\label{eq:interface_transmission_conditions_vu}
v_1=v_2, \qquad c_1^2\nabla u_1\cdot\mathbf{n}_1 + c_2^2\nabla u_2\cdot\mathbf{n}_2 = 0 \qquad \text{on }\Gamma_I.
\end{equation}
When $\Gamma_I$ is aligned with the computational mesh, the standard EDG formulation can be applied independently in the two subdomains, with the numerical traces coupled across the interface through \eqref{eq:interface_transmission_conditions_vu}. Our interest here is in the more general case in which $\Gamma_I$ intersects the background mesh arbitrarily. We therefore introduce the corresponding CutEDG discretization below.

\subsection{CutEDG interface formulation} \label{subsec:cutedg_interface_formulation}
The unfitted discretization is constructed separately on the active mesh associated with each subdomain. Since the geometric setting and ghost-penalty stabilization are analogous to those introduced for the single-domain problem, we adopt the same notation with an additional subscript $i$ to indicate the corresponding subdomain. For convenience, let $K_{\Omega_i}=K\cap\Omega_i$. For $i=1,2$, we define the active mesh and active domain by
\begin{equation}\label{eq:interface_active_mesh} 
\mathcal{T}_{h,i} = \left\{K\in\widetilde{\mathcal{T}}_h:|K_{\Omega_i}|>0\right\},
\qquad \Omega_{h,i} = \bigcup_{K\in\mathcal{T}_{h,i}}K.
\end{equation}
The corresponding set of cut cells is
\begin{equation}\label{eq:interface_cut_cells}
\mathcal{T}_{h,i}^{\mathrm{cut}} = \left\{K\in\mathcal{T}_{h,i}:0<|K_{\Omega_i}|<|K|\right\}.
\end{equation}
Thus, $\mathcal{T}_{h,i}^{\mathrm{cut}}$ contains cells intersected by the internal interface $\Gamma_I$ as well as, when applicable, cells intersected by the exterior boundary $\Gamma = \partial\Omega$; see Figure \ref{fig:interface} for an illustration in two dimensions.

\begin{figure}[htbp]
\centering
\begin{tikzpicture}[scale=0.9]
    % 网格参数：8×6 个单元
    \def\nx{8}
    \def\ny{6}
    \pgfmathsetmacro{\dx}{6/\nx}
    \pgfmathsetmacro{\dy}{4/\ny}
    
    % ---------- 1. 背景网格（灰色虚线） ----------
    \foreach \i in {0,...,\nx} {
        \draw[gray!50, dashed, thin] (\i*\dx,0) -- (\i*\dx,\ny*\dy);
    }
    \foreach \j in {0,...,\ny} {
        \draw[gray!50, dashed, thin] (0,\j*\dy) -- (\nx*\dx,\j*\dy);
    }
    
    % ---------- 2. 子区域填充 ----------
    % 圆心和半径
    \def\cx{3.0}
    \def\cy{2.0}
    \def\R{1.8}
    
    % 先填充 Ω2（外部）为极淡灰色，以示区分
    \fill[gray!5, opacity=0.2] (0,0) rectangle (\nx*\dx,\ny*\dy);
    % 再填充 Ω1（内部）为淡蓝色，覆盖 Ω2
    \fill[blue!10, opacity=0.3, even odd rule] (\cx,\cy) circle (\R);
    
    % 单独填充切割单元（使其颜色与 Ω1 一致）
    \foreach \i/\j in {2/2,2/3,3/1,3/2,3/3,3/4,4/1,4/2,4/3,4/4,5/2,5/3} {
        \pgfmathsetmacro{\x}{\i*\dx}
        \pgfmathsetmacro{\y}{\j*\dy}
        \fill[blue!10, opacity=0.3] (\x,\y) rectangle (\x+\dx,\y+\dy);
    }

    % ---------- 5. 高亮切割单元（斜线 + 粗边框） ----------
    \foreach \i/\j in {2/2,2/3,3/1,3/2,3/3,3/4,4/1,4/2,4/3,4/4,5/2,5/3} {
        \pgfmathsetmacro{\x}{\i*\dx}
        \pgfmathsetmacro{\y}{\j*\dy}
        \fill[pattern=north east lines, pattern color=blue!30] (\x,\y) rectangle (\x+\dx,\y+\dy);
        \draw[thick, blue!60] (\x,\y) rectangle (\x+\dx,\y+\dy);
    }

    % 高亮\Omega2区域
    \foreach \i/\j in {0/0,0/1,0/2,0/3,0/4,0/5,1/0,1/5,6/0,6/5,7/0,7/1,7/2,7/3,7/4,7/5} {
        \pgfmathsetmacro{\x}{\i*\dx}
        \pgfmathsetmacro{\y}{\j*\dy}
        \fill[pattern=north east lines, pattern color=green!30] (\x,\y) rectangle (\x+\dx,\y+\dy);
        \draw[thick, green!60] (\x,\y) rectangle (\x+\dx,\y+\dy);
    }

    % 高亮相交区域
    % 高亮\Omega2区域
    \foreach \i/\j in {1/1,1/2,1/3,1/4,2/0,2/1,2/4,2/5,3/0,3/5,4/0,4/5,5/0,5/1,5/4,5/5,6/1,6/2,6/3,6/4} {
        \pgfmathsetmacro{\x}{\i*\dx}
        \pgfmathsetmacro{\y}{\j*\dy}
        \fill[pattern=north east lines, pattern color=pink!40] (\x,\y) rectangle (\x+\dx,\y+\dy);
        \draw[thick, pink!60] (\x,\y) rectangle (\x+\dx,\y+\dy);
    }

    % ---------- 3. 绘制界面 Γ（红色实线） ----------
    \draw[thick, red] (\cx,\cy) circle (\R) node[above right, xshift=0.8cm, yshift=0.7cm, black] {$\Gamma_I$};

    % ---------- 4. 标注子区域 ----------
    \node at (\cx-0.2,\cy-0.2) {$\Omega_1$};   % 内部
    \node at (5.0,3.5) {$\Omega_2$};
    \node at (0.4,0.4) {$\Gamma$};
    
    % ---------- 7. 坐标轴 ----------
    \draw[->] (-0.5,0) -- (6.8,0) node[right] {$x$};
    \draw[->] (0,-0.5) -- (0,4.8) node[above] {$y$};
    
\end{tikzpicture}
\caption{The active mesh $\mathcal{T}_{h,i}$. Subdomains $\Omega_i$ are separated by the immersed interface $\Gamma_I$, $\Gamma$ is the exterior boundary, and the red region corresponds to the cut elements $\mathcal{T}_{h,i}^{\mathrm{cut}}$.}
\label{fig:interface}
\end{figure}

To construct the EDG formulation, we approximate $u_i$ and $v_i$ by piecewise polynomials of degrees $p$ and $q$, respectively. Precisely, we define
\begin{equation}\label{eq:approximation_space_interface}
\begin{aligned}
V_{h,i}^{r} &:= \{w_h({\bf x})\big| w_h({\bf x}) \in \mathcal{P}^{r}(K), \,{\bf x} \in K\quad \forall K \in \mathcal{T}_{h,i} \}.
\end{aligned}
\end{equation}
For each $i=1,2$, the CutEDG approximation seeks $(u_{h,i},v_{h,i})\in V_{h,i}^{p}\times V_{h,i}^{q}$ such that, for all $(\phi_{u_i},\phi_{v_i},\widetilde{\phi}_i)\in V_{h,i}^{p}\times V_{h,i}^{q}\times V_{h,i}^0$, the mean constraint
\begin{equation}\label{eq:interface_cutEDG_mean}
\sum_{K\in\mathcal{T}_{h,i}}\int_{K_{\Omega_i}}h^{-2}\big((u_{h,i})_t-v_{h,i}\big)\widetilde{\phi}_i\,d\mathbf{x}+\gamma_M^{(u_i)}h^{-2}J_{p,i}\big((u_{h,i})_t,\widetilde{\phi}_i\big)=0,
\end{equation}
and the energy equations
\begin{subequations}\label{eq:interface_cutEDG}
\begin{align}
&\sum_{K\in\mathcal{T}_{h,i}}\int_{K_{\Omega_i}}c_i^2\nabla\big((u_{h,i})_t-v_{h,i}\big)\cdot\nabla\phi_{u_i}\,d\mathbf{x}+\gamma_M^{(u_i)}h^{-2}J_{p,i}\big((u_{h,i})_t,\phi_{u_i}\big)
\nonumber\\
=&\sum_{K\in\mathcal{T}_{h,i}}\int_{\partial K_{\Omega_i}}c_i^2\big(\widehat{v_{h,i}}-v_{h,i}\big)\nabla\phi_{u_i}\cdot\mathbf{n}\,dS,\label{eq:interface_cutEDG_u}\\
&\sum_{K\in\mathcal{T}_{h,i}}\int_{K_{\Omega_i}}(v_{h,i})_t\phi_{v_i}+ c_i^2\nabla u_{h,i}\cdot\nabla\phi_{v_i}\,d\mathbf{x}+\gamma_M^{(v_i)}J_{q,i}\big((v_{h,i})_t,\phi_{v_i}\big)
\nonumber\\
= &\sum_{K\in\mathcal{T}_{h,i}}\int_{\partial K_{\Omega_i}}\widehat{c_i^2\nabla u_{h,i}}\cdot\mathbf{n}\,\phi_{v_i}\,dS
+\sum_{K\in\mathcal{T}_{h,i}}\int_{K_{\Omega_i}}f_i\phi_{v_i}\,d\mathbf{x}
\label{eq:interface_cutEDG_v}
\end{align}
\end{subequations}
are satisfied. Here, $\gamma_M^{(u_i)}>0$ and $\gamma_M^{(v_i)}>0$ are stabilization parameters. For a polynomial degree $r$, the ghost-penalty form in subdomain $\Omega_i$ is defined by
\begin{equation}\label{eq:interface_ghost_penalty}
J_{r,i}(w,z)=\sum_{F\in\mathcal{F}_{h,i}^{g}}\sum_{\ell=0}^{r}\omega_\ell\frac{h^{2\ell+1}}
{(2\ell+1)(\ell!)^2}\int_F [\![\partial_{n_F}^{\ell}w]\!][\![\partial_{n_F}^{\ell}z]\!]
\,dS,
\end{equation}
where the corresponding set of ghost-penalty faces is
\begin{equation}\label{eq:interface_ghost_faces}
\mathcal{F}_{h,i}^{g} = \left\{
\begin{aligned}
F:\;&F\text{ is a common $(d-1)$-dimensional face of }
K^+,K^-\in\mathcal{T}_{h,i},\\
&K^+\in\mathcal{T}_{h,i}^{\mathrm{cut}} \text{ or } K^-\in\mathcal{T}_{h,i}^{\mathrm{cut}}
\end{aligned}
\right\}.
\end{equation}

It remains to specify the numerical fluxes $\widehat{v_{h,i}}$ and $\widehat{c_i^2\nabla u_{h,i}}$. On interior mesh faces contained in $\Omega_i$, we employ the same EDG numerical fluxes as in \eqref{eq:EDG_fluxes}, with the spatial flux $c_i^2\nabla u_{h,i}$ replacing $\nabla u_h$. On the exterior boundary $\partial\Omega$, the boundary fluxes are imposed as in the single-domain formulation. On the internal interface $\Gamma_I$, the two subdomain solutions are coupled through the transmission conditions \eqref{eq:interface_transmission_conditions_vu}. In particular, the numerical trace of $v$ is chosen to be single-valued across the interface,
\begin{equation}\label{eq:interface_numerical_flux_v}
\widehat{v_{h,1}} = \widehat{v_{h,2}} \qquad \text{on }\Gamma_I,
\end{equation}
whereas conservation of the numerical spatial flux requires
\begin{equation}\label{eq:interface_numerical_flux_grad}
\widehat{c_1^2\nabla u_{h,1}}\cdot\mathbf{n}_1 + \widehat{c_2^2\nabla u_{h,2}}\cdot\mathbf{n}_2 = 0 \qquad \text{on }\Gamma_I.
\end{equation}

To define an energy-stable interface flux, we fix the orientation $\mathbf{n}=\mathbf{n}_1=-\mathbf{n}_2$ and introduce the one-sided normal fluxes
\begin{equation}\label{eq:interface_normal_flux}
q_{h,i}=c_i^2\nabla u_{h,i}\cdot\mathbf{n},\qquad i=1,2.
\end{equation}
With this common orientation, the exact transmission conditions are equivalent to
\[v_1=v_2,\qquad q_1=q_2 \qquad \text{on }\Gamma_I.\]
where $q_i = c_i^2\nabla u_i\cdot\mathbf{n}$. Then we define 
\begin{equation}\label{eq:interface_fluxes}
\begin{aligned}
\widehat{v}&=\alpha v_{h,1}+(1-\alpha)v_{h,2},\\
\widehat{q}&=(1-\alpha)q_{h,1}+\alpha q_{h,2},
\end{aligned}
\end{equation}
where $\alpha\in[0,1]$. The numerical traces on the two sides of the interface are consequently assigned as
\begin{equation}\label{eq:interface_flux_assignment}
\widehat{v_{h,1}}=\widehat{v_{h,2}}=\widehat{v}, \qquad \widehat{c_1^2\nabla u_{h,1}}\cdot\mathbf{n}_1 = \widehat{q},
\qquad \widehat{c_2^2\nabla u_{h,2}}\cdot\mathbf{n}_2 = -\widehat{q}.
\end{equation}
Thus, the interface conservation conditions
\eqref{eq:interface_numerical_flux_v}--\eqref{eq:interface_numerical_flux_grad}
are satisfied by construction.

\subsection{Energy stability}\label{subsec:interface_stability}

The stability of the interface formulation follows from the same energy argument used for the single-domain CutEDG method. The only additional contribution arises from the internal interface, where the numerical fluxes in \eqref{eq:interface_fluxes} provide the required dissipation.

\begin{theorem}\label{thm:interface_stability}
Assume that $f_i=0$, and that homogeneous boundary conditions are imposed on $\partial\Omega$. Let the interface numerical fluxes be given by \eqref{eq:interface_fluxes}, with $\tau,\beta\geq0$. Then the semi-discrete CutEDG interface formulation \eqref{eq:interface_cutEDG_mean}--\eqref{eq:interface_cutEDG} satisfies
\begin{equation}\label{eq:interface_energy_stability}
\frac{d}{dt}\left(E_h^I(t) +\sum_{i=1}^2c_i^2\gamma_M^{(u_i)}h^{-2}J_{p,i}(u_{h,i},u_{h,i})+\gamma_M^{(v_i)}J_{q,i}(v_{h,i},v_{h,i})\right)\leq0,
\end{equation}
where the discrete energy is defined by
\begin{equation}\label{eq:interface_discrete_energy}
E_h^I(t)=\frac{1}{2}\sum_{i=1}^{2}\sum_{K\in\mathcal{T}_{h,i}}\int_{K\cap\Omega_i}c_i^2|\nabla u_{h,i}|^2+|v_{h,i}|^2 \,d\mathbf{x}.
\end{equation}
\end{theorem}

\begin{proof}
For each $i=1,2$, we choose $\phi_{u_i}=u_{h,i}, \phi_{v_i}=v_{h,i}$ in \eqref{eq:interface_cutEDG_u}--\eqref{eq:interface_cutEDG_v} and sum the result equaitons. Using the symmetry of the ghost-penalty forms,
\[
J_{p,i}((u_{h,i})_t,u_{h,i}) = \frac{1}{2} \frac{d}{dt} J_{p,i}(u_{h,i},u_{h,i}),\qquad J_{q,i}((v_{h,i})_t,v_{h,i})=\frac{1}{2}\frac{d}{dt}J_{q,i}(v_{h,i},v_{h,i}).
\]
The contributions from interior faces in each subdomain are treated exactly as in the single-domain stability analysis. Consequently,
\begin{equation}\label{eq:dE_interface}
\begin{aligned}
&\frac{d}{dt}\left(E_h^I(t) +\sum_{i=1}^2c_i^2\gamma_M^{(u_i)}h^{-2}J_{p,i}(u_{h,i},u_{h,i})+\gamma_M^{(v_i)}J_{q,i}(v_{h,i},v_{h,i})\right) \\
= & -\sum_{i=1}^{2} \sum_{F\in\mathcal{F}_h^{\Omega_i}}  \int_F \tau\left|[\![\nabla u_{h,i}]\!]\right|^2  + \beta 
\left|[\![ v_{h,i}]\!]\right|^2 \,dS \\
&+ \underbrace{ \int_{\Gamma_I}  \widehat{v}(q_{h,1}-q_{h,2}) + \widehat{q}(v_{h,1}-v_{h,2}) - v_{h,1}q_{h,1} + v_{h,2}q_{h,2} \,dS}_{\mathcal{D}_{\Gamma_I}}.
\end{aligned}
\end{equation}
 Using the interface fluxes \eqref{eq:interface_fluxes}, we obtain $\mathcal{D}_{\Gamma_I} = 0$. Substitution into \eqref{eq:dE_interface} therefore yields \eqref{eq:interface_energy_stability}, which proves the result.
\end{proof}

\begin{remark}
The a priori error analysis for the interface problem can be established
using the same arguments as in the single-domain case. Since the interface
coupling does not require any essential modifications to the error analysis,
the proof is omitted for brevity.
\end{remark}

%%%%%%%%%%%%%%%%%%%%%%%%%%%
\section{Numerical experiments} \label{sec:num}

In this section, we present numerical experiments for two-dimensional boundary-cut and interface problems to assess the accuracy and robustness of the proposed CutEDG method and validate the theoretical results. Time integration is performed using the third-order strong stability-preserving Runge--Kutta (SSPRK3) method. To minimize the influence of temporal discretization errors on the observed spatial convergence rates, we choose $\Delta t=\frac{1}{25}h^{(p+1)/3}$, where $p$ denotes the polynomial degree of the DG approximation space for $u_h$. This choice ensures that the temporal discretization error is dominated by the spatial discretization error.  Finally, to evaluate the numerical integration over irregular domains, we use the $deal.II$ library \cite{bangerth2016deal} to implement the method.  

\begin{remark}
    Unless otherwise specified, all numerical experiments employ the alternating flux with $\alpha=\beta=\tau=0$, for which the optimal convergence rate of order $p+1$ is observed. We also investigate the Sommerfeld flux ($\alpha=\beta=\tau=\frac12$) and the central flux ($\alpha=\frac12$, $\beta=\tau=0$). The observed convergence behavior for these alternative flux choices is consistent with that reported for the original EDG method \cite{appelo2015new}.  
\end{remark}

%%%%%%%%%%%%%%%%%%%%%%%%
\begin{example}\label{ex1} \textbf{Convergence test.}
\end{example}

We first examine the accuracy and convergence properties of the proposed CutEDG method for the two-dimensional scalar wave equation
\begin{equation}\label{equation2}
    u_{tt} = u_{xx} + u_{yy}.
\end{equation}
Two computational domains are considered to assess the performance of the method for different geometries and cut-cell configurations.

\begin{itemize}
\item[(i)] \textbf{Rectangular domain.} We consider $\Omega = [-\pi,\pi]^2$, with initial and boundary conditions prescribed according to the exact solution
$u(x,y,t) = \sin(x)\sin(y)\cos(\sqrt{2}t).$ 
The background mesh is arranged so that small cut cells are generated along the left and bottom boundaries; see Figure~\ref{fig:cutconfig_a}. Their relative sizes in the \(x\)- and \(y\)-directions are characterized by the parameters \(\epsilon_x\) and \(\epsilon_y\), respectively.

\item[(ii)] \textbf{Circular domain.} We consider the unit disk $\Omega=\{(x,y)\in \mathbb{R}^2: x^2+y^2\leq 1\}$ with homogeneous Dirichlet boundary conditions and initial conditions corresponding to the exact solution
$ u(x,y,t) = B_0\left( \alpha_3 \sqrt{x^2+y^2} \right)\cos(\alpha_3 t)$,
where $B_0$ is the zeroth-order Bessel function of the first kind, and
$\alpha_n$ denote its $n$-th positive zero. The fictitious domain $\Omega'=[-1.1,1.1]^2$ is discretized using a uniform Cartesian background mesh (see Figure \ref{fig:cutconfig_b}). 
\end{itemize}
For both configurations, the final time is set to $T = 0.25$, and the polynomial degree $p = 2, 3, 4$ are considered with $q = p-1$.

The $L^2$-errors of the numerical solution $u_h$ for the rectangular and circle domains are presented in Figure \ref{fig:ex1_rect_accuacy} and the left panel of Figure \ref{fig:ex1_circle_accuacy}, showing that the proposed CutEDG method achieves the optimal convergence rate of $p+1$ for all polynomial degrees considered. These results demonstrate that the method preserves its high-order accuracy on unfitted meshes, even in the presence of small cut cells and curved boundaries.

To further assess the effectiveness of the stabilization, the mass-matrix condition numbers are examined for different mesh resolutions and polynomial degrees in the circular domain, see the right panel of Figure \ref{fig:ex1_circle_accuacy} and Table \ref{Table:cond_2d}. We observe that the condition numbers obtained with the proposed CutEDG method are of the same order of magnitude as those reported for the high-order continuous CutFEM in \cite{sticko2019higher} with slightly smaller values for the polynomial degrees considered. These results indicate that the stabilization effectively mitigates the ill-conditioning associated with small cut cells, while maintaining mass-matrix conditioning comparable to that of the stabilized continuous CutFEM formulation. Nevertheless, the condition numbers increase with the polynomial degree $p$, indicating that the polynomial-degree dependence of the mass-matrix conditioning persists despite the stabilization.

%----------------------------- colours ----------------------------
\definecolor{cDash}{gray}{0.70}          % dashed background grid of (a)
\definecolor{cSolid}{gray}{0.45}         % solid mesh of (a)
\definecolor{cNode}{RGB}{228,50,35}      % mesh vertices of (a)

\definecolor{cgFill}{RGB}{195,253,185}\definecolor{cgLine}{RGB}{136,232,116}
\definecolor{crFill}{RGB}{244,181,180}\definecolor{crLine}{RGB}{228,130,128}
\definecolor{cbFill}{RGB}{178,178,249}\definecolor{cbLine}{RGB}{138,138,235}
%         ^ fill of the cells   ^ colour of the cell edges
%         green = completely inside, red = cut, blue = completely outside

\tikzset{
  bgGrid/.style   = {draw=cDash,  line width=0.35pt, dash pattern=on 3.2pt off 2.4pt},
  meshLine/.style = {draw=cSolid, line width=0.60pt},
  meshNode/.style = {circle, fill=cNode, draw=none, inner sep=0pt, minimum size=4.2pt},
  cellEdge/.style = {line width=0.30pt},
}

%----------------------------- PARAMETERS -----------------------------
% ---- overall layout ----
\def\figscale{2.0}                 
\def\figbox{(-1.60,-1.55) rectangle (1.16,1.16)} % SHARED bounding box -> alignment
\def\labfont{\footnotesize}               
\def\aL{0.6}      
\def\aR{5}        
% ---- (b) circle domain ----
\def\hB{0.11}     
\def\NB{20}       
\def\NBm{19}      
\def\RB{1}        
\def\boxB{1.1}   

\begin{figure}[t]
  \centering
  %======================= (a) 
  \begin{subfigure}[t]{0.48\linewidth}
    \centering
    \begin{tikzpicture}[scale=\figscale]
      \path[use as bounding box] \figbox;
      \begin{scope}[shift={(-1.1,-1.1)}, scale=0.36666667]
        \draw[bgGrid] (0,0) grid (6,6);
        \draw[meshLine] (\aL,\aL) rectangle (\aR,\aR);
      \end{scope}
    \end{tikzpicture}
    \caption{Rectangular domain.}
    \label{fig:cutconfig_a}
  \end{subfigure}\hfill
  %========================= (b) 
  \begin{subfigure}[t]{0.48\linewidth}
    \centering
    \begin{tikzpicture}[scale=\figscale]
      \path[use as bounding box] \figbox;
      \foreach \i in {0,...,\NBm}{%
        \foreach \j in {0,...,\NBm}{%
          \pgfmathsetmacro\xa{-\boxB+\i*\hB}%
          \pgfmathsetmacro\xb{-\boxB+(\i+1)*\hB}%
          \pgfmathsetmacro\ya{-\boxB+\j*\hB}%
          \pgfmathsetmacro\yb{-\boxB+(\j+1)*\hB}%
          \pgfmathsetmacro\qa{\xa*\xa+\ya*\ya}%
          \pgfmathsetmacro\qb{\xa*\xa+\yb*\yb}%
          \pgfmathsetmacro\qc{\xb*\xb+\ya*\ya}%
          \pgfmathsetmacro\qd{\xb*\xb+\yb*\yb}%
          \pgfmathsetmacro\qmax{max(max(\qa,\qb),max(\qc,\qd))}%
          \pgfmathsetmacro\qmin{min(min(\qa,\qb),min(\qc,\qd))}%
          \pgfmathtruncatemacro\ct{\qmax<=\RB*\RB ? 0 : (\qmin>\RB*\RB ? 2 : 1)}%
          \ifnum\ct=0 \def\cF{cgFill}\def\cL{cgLine}\fi
          \ifnum\ct=1 \def\cF{crFill}\def\cL{crLine}\fi
          \ifnum\ct=2 \def\cF{cbFill}\def\cL{cbLine}\fi
          \fill[\cF] (\xa,\ya) rectangle (\xb,\yb);
          \draw[cellEdge,\cL] (\xa,\ya) rectangle (\xb,\yb);
        }}%
      \draw[line width=1.1pt, black] (0,0) circle[radius=\RB];
      \foreach \t in {-1.1,-0.6,-0.1,0.4,0.9}{%
        \node[anchor=north, inner sep=1pt, font=\labfont] at (\t,-1.14) {$\t$};
        \node[anchor=east,  inner sep=1pt, font=\labfont] at (-1.14,\t) {$\t$};}
      \node[anchor=north, inner sep=1pt, font=\labfont] at (0,-1.44)     {$x$};
      \node[anchor=east,  inner sep=1pt, font=\labfont] at (-1.47,0)     {$y$};
    \end{tikzpicture}
    \caption{Circle domain.}
    \label{fig:cutconfig_b}
  \end{subfigure}
  \caption{Example \ref{ex1}: Illustration of the two mesh configurations.}
\end{figure}

\begin{figure}[!ht]
\centering 
\includegraphics[width=0.32\textwidth,trim={0.5cm 0.5cm 1cm 0.5cm},clip]{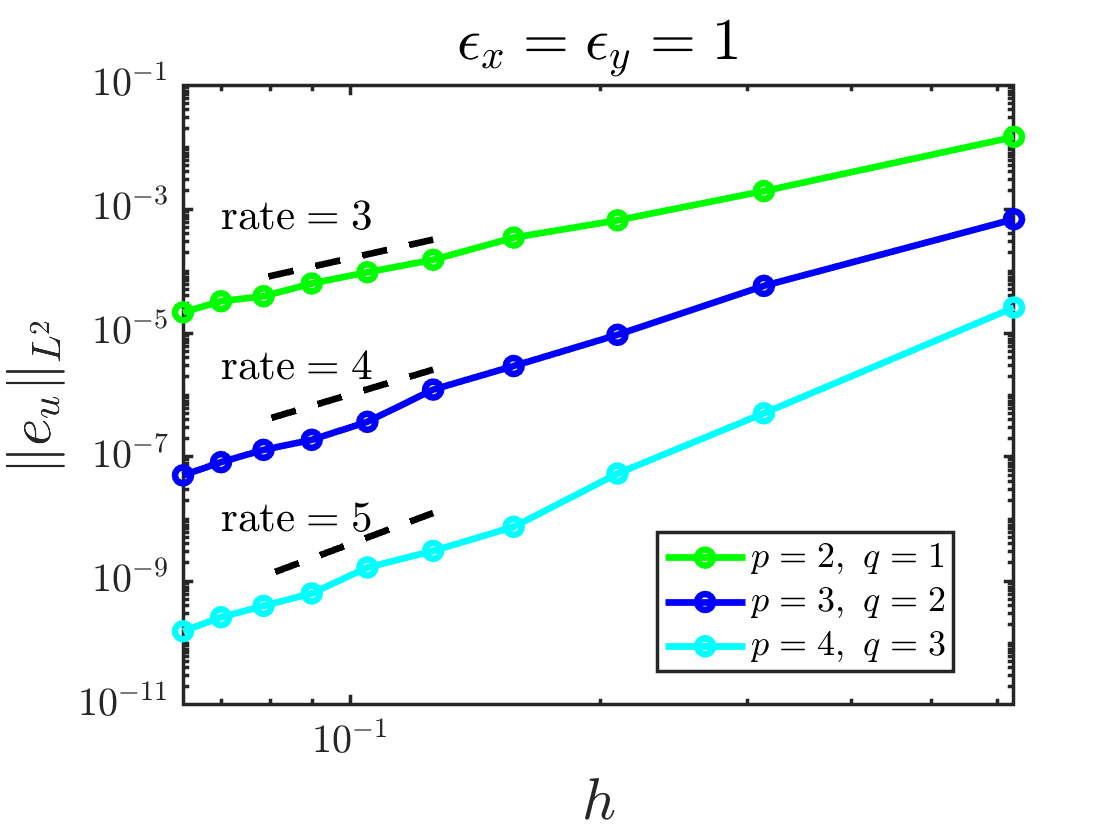}
\includegraphics[width=0.32\textwidth,trim={0.5cm 0.5cm 1cm 0.25cm},clip]{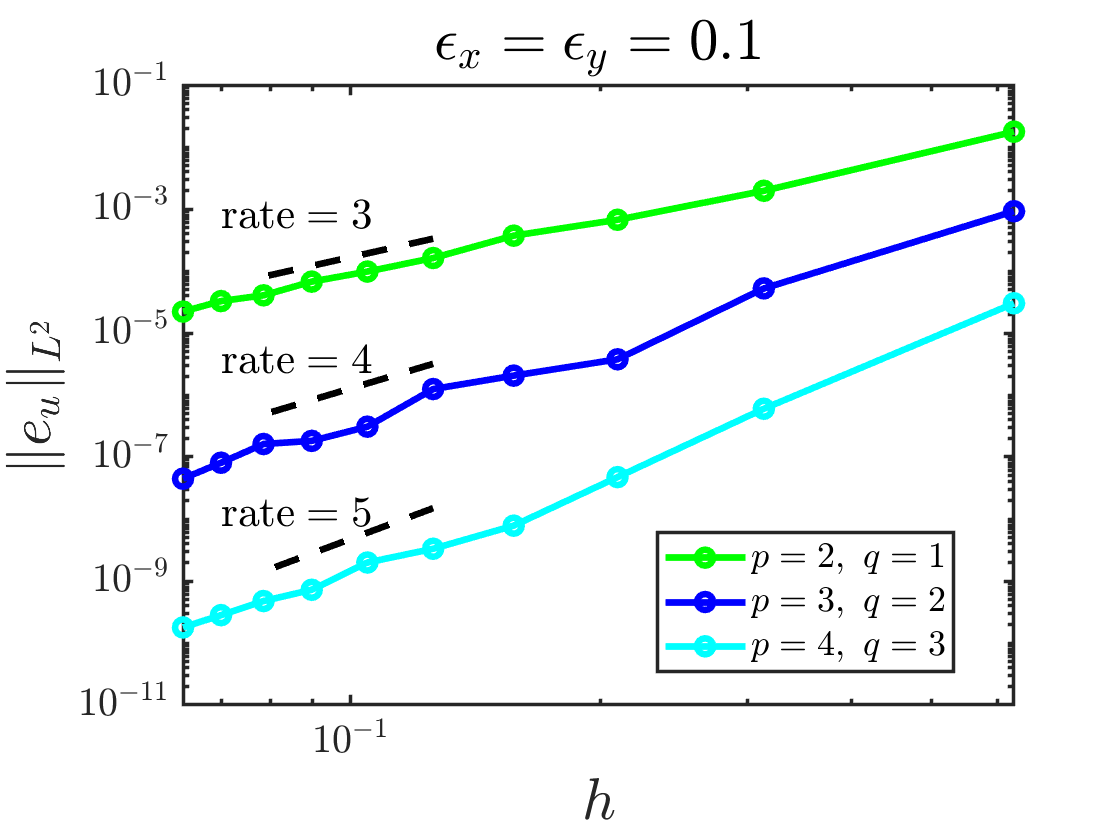}
\includegraphics[width=0.32\textwidth,trim={0.5cm 0.5cm 1cm 0.25cm},clip]{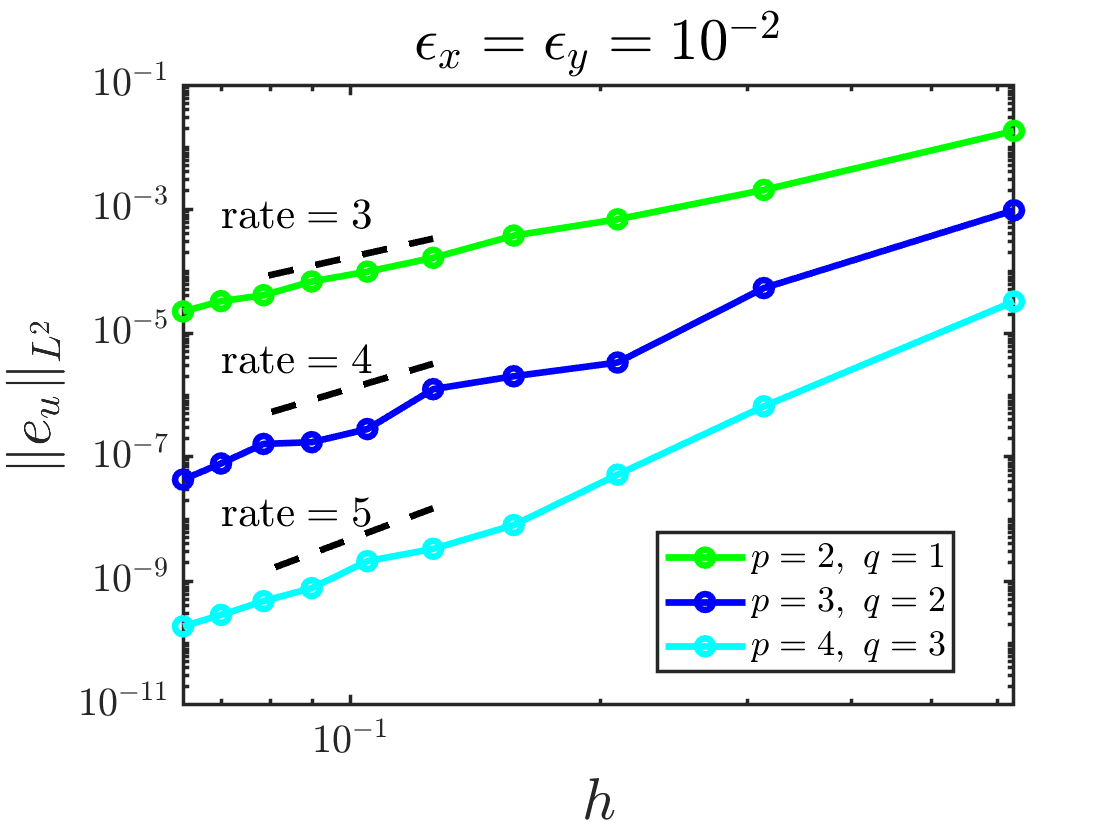}\\
\includegraphics[width=0.32\textwidth,trim={0.5cm 0.5cm 1cm 0.3cm},clip]{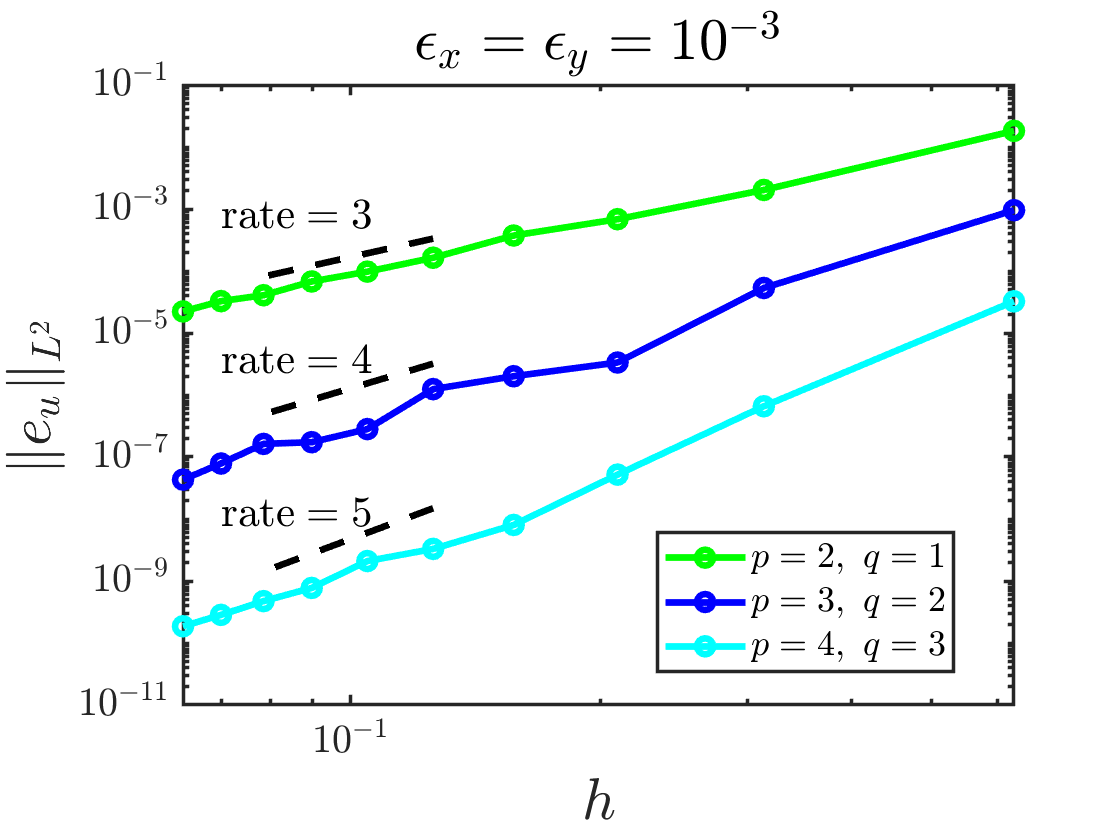}
\includegraphics[width=0.32\textwidth,trim={0.5cm 0.5cm 1cm 0.3cm},clip]{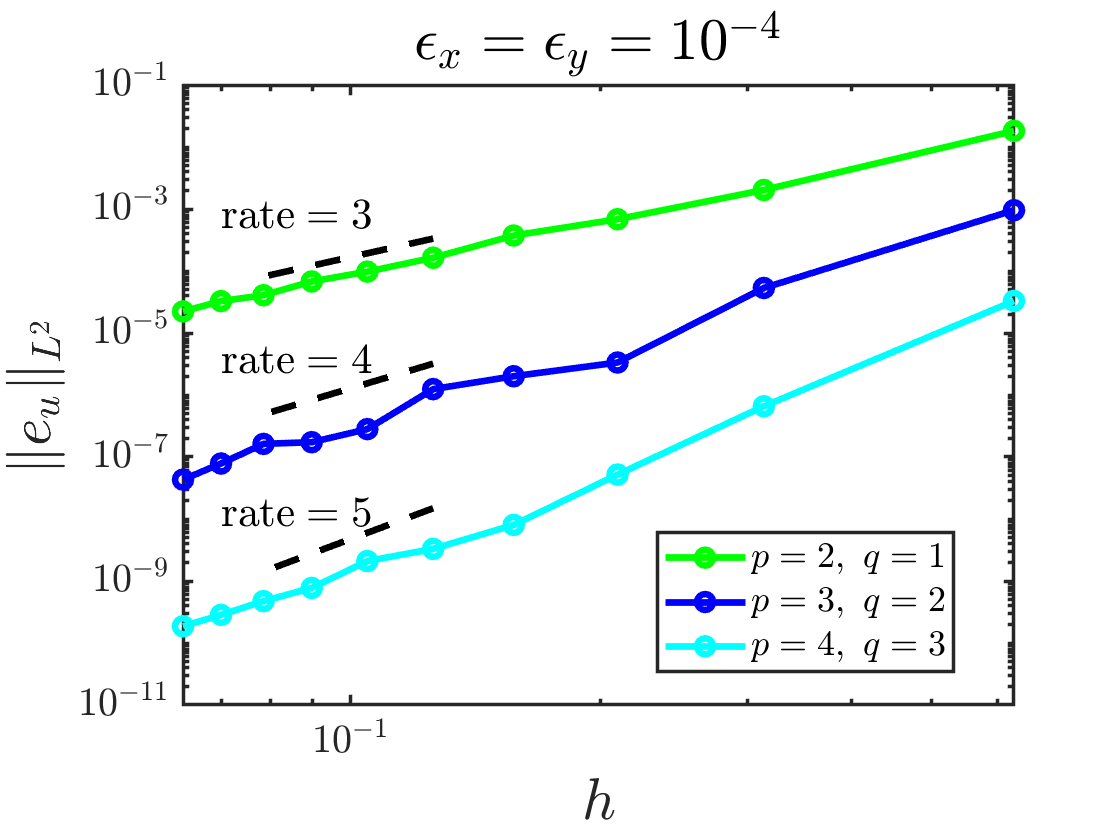}
\includegraphics[width=0.32\textwidth,trim={0.5cm 0.5cm 1cm 0.3cm},clip]{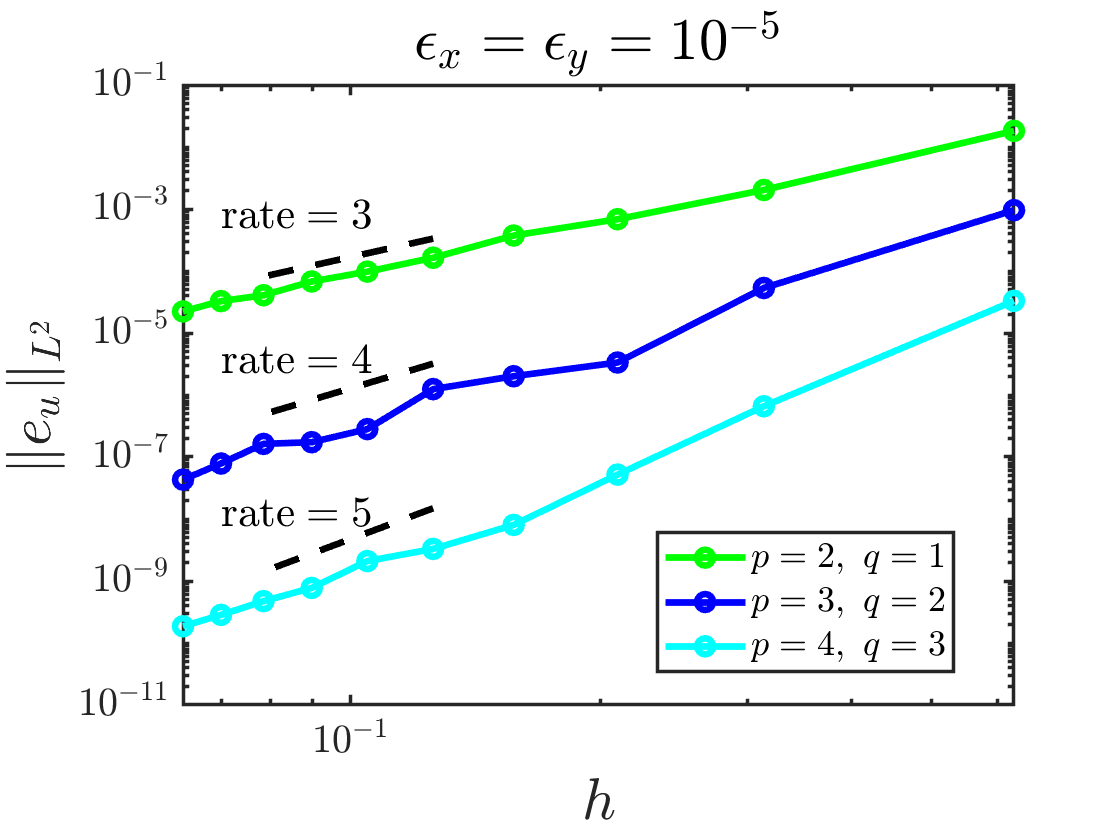}
\caption{Example \ref{ex1}: $L^2$ errors for different cut sizes in the rectangular domain at $T = 0.25$.}
\label{fig:ex1_rect_accuacy} 
\end{figure}

\begin{figure}[!ht]
\begin{center}
\includegraphics[width=0.35\textwidth]{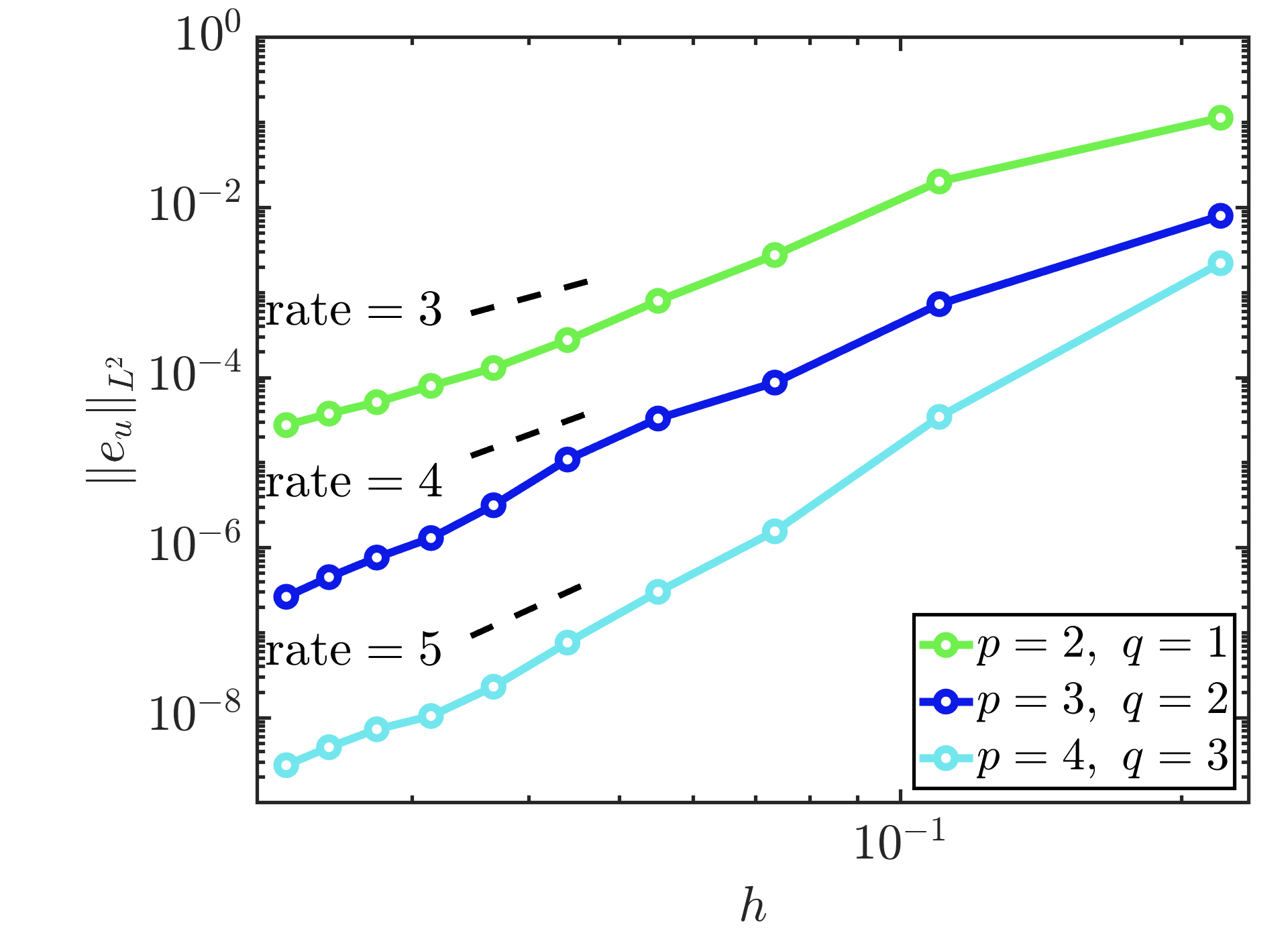}
\includegraphics[width=0.35\textwidth]{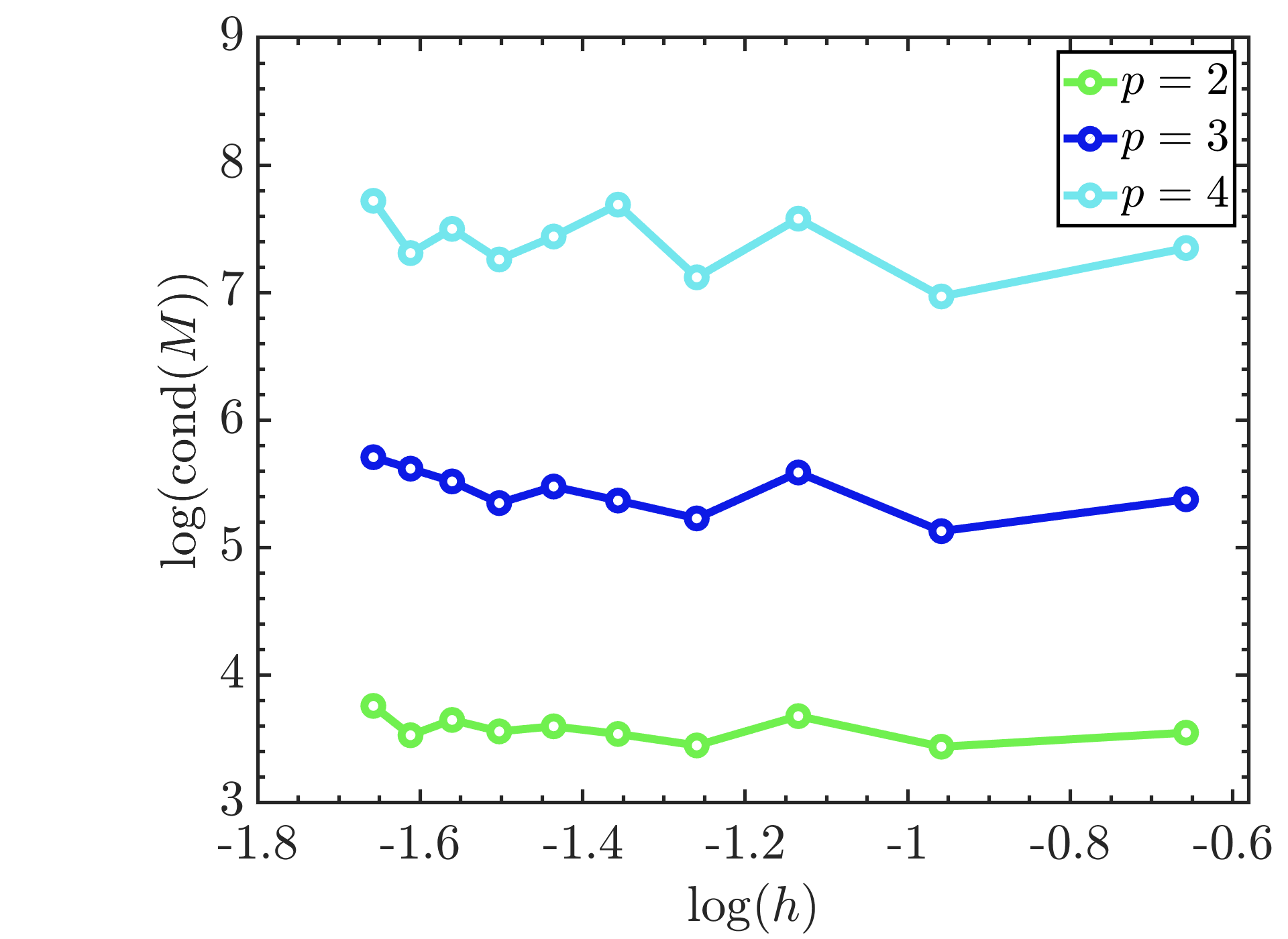}
\end{center}
\caption{Example \ref{ex1}. \textbf{Left:} $L^2$-errors of the CutEDG method in the circular domain at $T = 0.25$. \textbf{Right:} Condition numbers of the mass matrices for different polynomial degrees.} 
\label{fig:ex1_circle_accuacy}
\end{figure}

\begin{table}[htbp]
  \centering
  \caption{Comparison of condition numbers of unstabilized and stabilized mass matrices for various polynomial degrees on meshes with $N_x = N_y = 20$ and $N_x = N_y = 100$.}
  \begin{tabular}{|c|cc|cc|}
    \hline
    & \multicolumn{2}{c|}{$N_x = N_y = 20$} & \multicolumn{2}{c|}{$N_x = N_y = 100$} \\
    \hline
    degree & $\mathcal{K}^{unstab}(M)$ & $\mathcal{K}^{stab}(M)$
           & $\mathcal{K}^{unstab}(M)$ & $\mathcal{K}^{stab}(M)$\\
    \hline
    $p = 2, q = 1$ & 1.29e+12 & 2.78e+4 & 2.39e+21  &3.49e+4 \\
    $p = 3, q = 2$ & 1.38e+19 & 1.37e+6 & 2.62e+22 &2.36e+6 \\
    $p = 4, q = 3$ & 3.11e+19 & 2.26e+8 & 3.23e+23 & 3.19e+8\\
    \hline
  \end{tabular}
\label{Table:cond_2d}
\end{table}

%%%%%%%%%%%%%%%%%%%%%%%%%%%%%
\begin{example}\label{ex2} \textbf{Wave propagation around an immersed obstacle.}
\end{example}
We next consider wave propagation in a domain containing an immersed obstacle, as illustrated in Figure \ref{fig:ex2_background}. The interior boundary $\Gamma_{int}$ is defined by the zero level set of the function 
\begin{equation*}
    \phi(x,y) = -\left(s^4(x^4 + y^4) + \dfrac{1}{4}s^2xy - 2s^2x^2 +2s^2y^2 - \frac{1}{4}\right),
\end{equation*}
with $s = 1.5$. The computational domain is defined as
\begin{equation}
    \Omega = \{(x,y)\in (-1.5,1.5)^2: \phi(x,y)<0\}.
\end{equation}
The initial conditions are prescribed as $u(x,y,0) = 0, u_t(x,y,0) = 0$. On the immersed boundary $\Gamma_{int}$, we impose the homogeneous Neumann condition, $\nabla u\cdot {\bf n} = 0$, where $\mathbf{n}$ denotes the unit outward normal to $\Gamma_{int}$. On the outer boundary, Dirichlet conditions are prescribed as
\begin{equation*}
u =
\begin{cases}
g_D(x, t), & x \in \Gamma_B, \\
0, & x \in \Gamma_L \cup \Gamma_R \cup \Gamma_U,
\end{cases}
\end{equation*}
where the time-dependent boundary excitation is given by
\begin{equation*}
    g_D(x,t) = \cos\left(\frac{\pi}{3}x\right)e^{-(t-t_c)^2/\sigma^2},\quad \sigma = 0.25, \quad t_c = 3.
\end{equation*}

In Figure \ref{fig:ex2_snapshot}, we show snapshots of the numerical solution at selected times to illustrate the propagation of waves around the immersed obstacle. Since an analytical solution is not available for this configuration, a numerical solution computed on a sufficiently refined mesh is used as the reference solution to assess the accuracy and convergence behavior of the proposed method. Figure \ref{fig:ex2_accuracy} presents the numerical errors obtained on a sequence of uniformly refined background meshes with resolutions ranging from $16\times16$ to $256\times256$. The errors decrease systematically under mesh refinement, and the observed convergence rates agree with the expected optimal orders. These results further demonstrate that the proposed CutEDG method retains its high-order accuracy for wave propagation problems involving complex immersed geometries.

\begin{figure}[!ht]
\begin{center}
\includegraphics[width=2.5in]{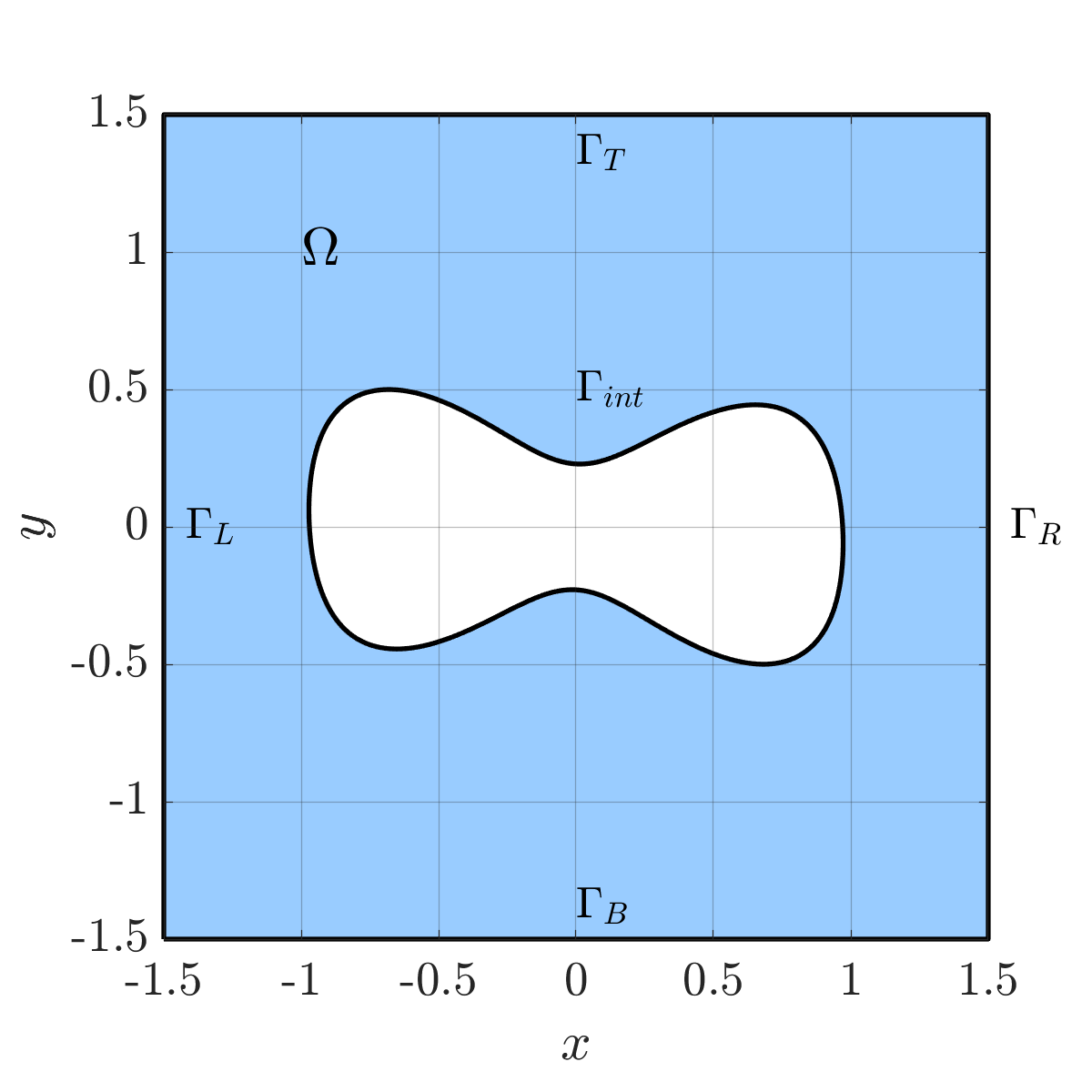}
\end{center}
\caption{Example \ref{ex2}: Uniform Cartesian background mesh on $[-1.5,1.5]^2$. The computational domain $\Omega$ is indicated by the blue region $(\phi<0)$, with homogeneous Neumann conditions imposed on the immersed boundary $\Gamma_{int}$ and Dirichlet conditions on the outer boundaries $\Gamma_L$, $\Gamma_R$, $\Gamma_T$, and $\Gamma_B$.}
\label{fig:ex2_background}
\end{figure}

% Neumann result
\begin{figure}[!ht]
\centering 
\includegraphics[width=0.32\textwidth,trim={2cm 0.3cm 2cm 0.3cm},clip]{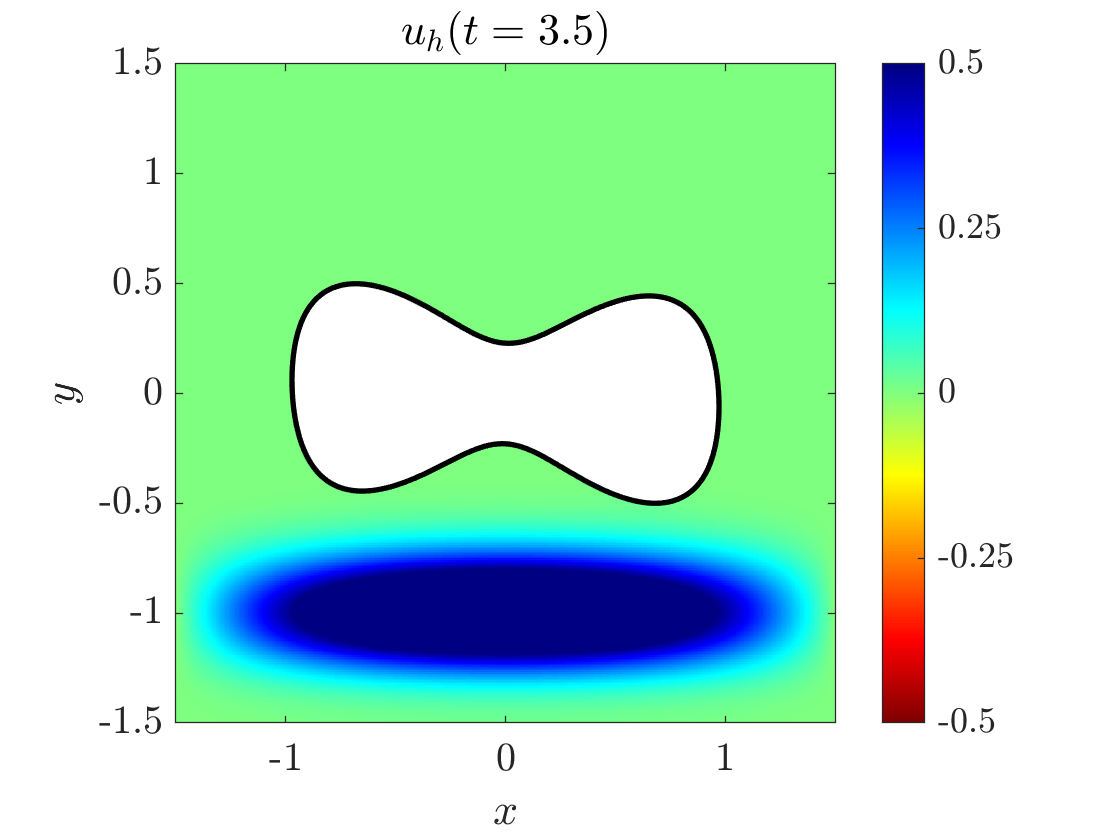}
\includegraphics[width=0.32\textwidth,trim={2cm 0.3cm 2cm 0.3cm},clip]{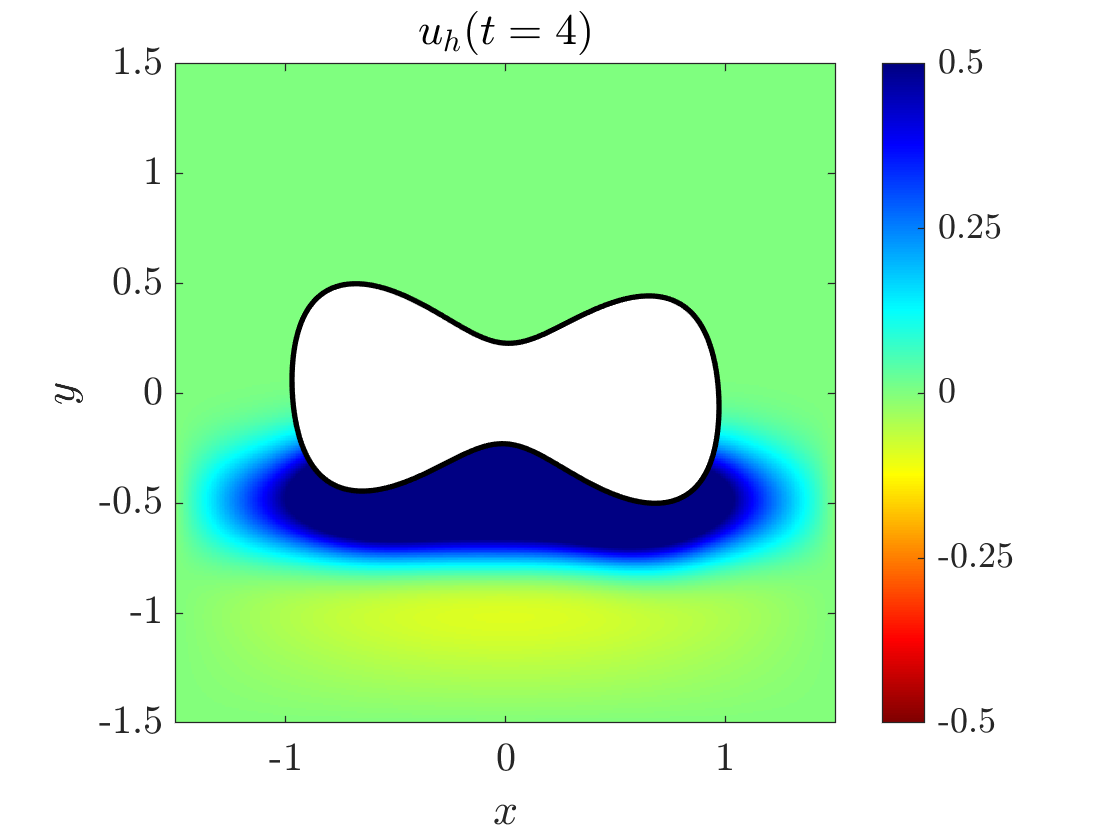}
\includegraphics[width=0.32\textwidth,trim={2cm 0.3cm 2cm 0.3cm},clip]{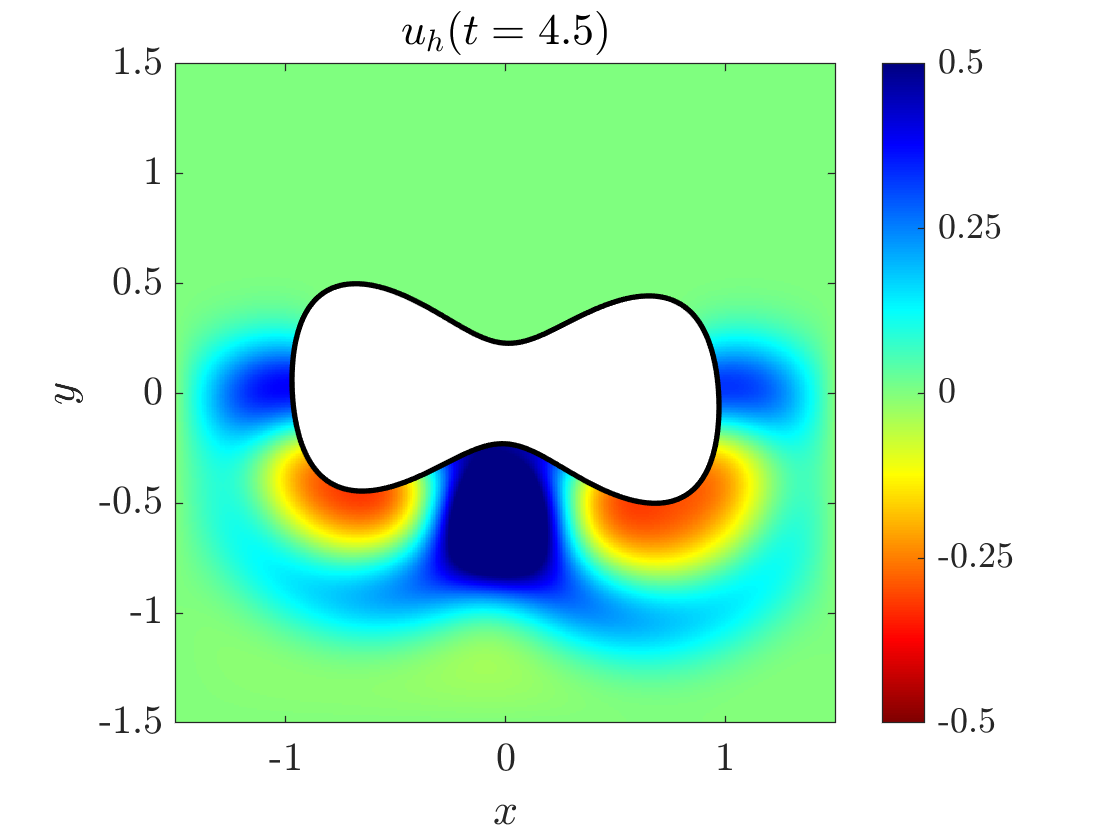}\\
\includegraphics[width=0.32\textwidth,trim={2cm 0.3cm 2cm 0.3cm},clip]{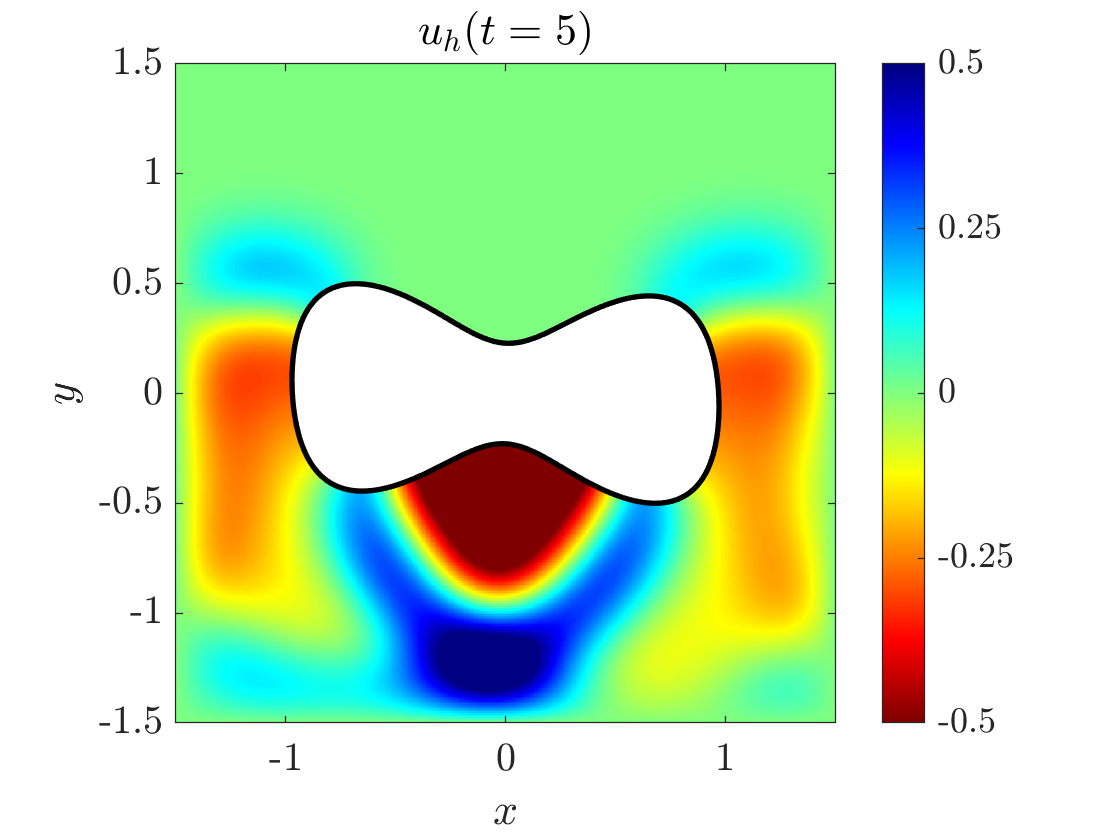}
\includegraphics[width=0.32\textwidth,trim={2cm 0.3cm 2cm 0.3cm},clip]{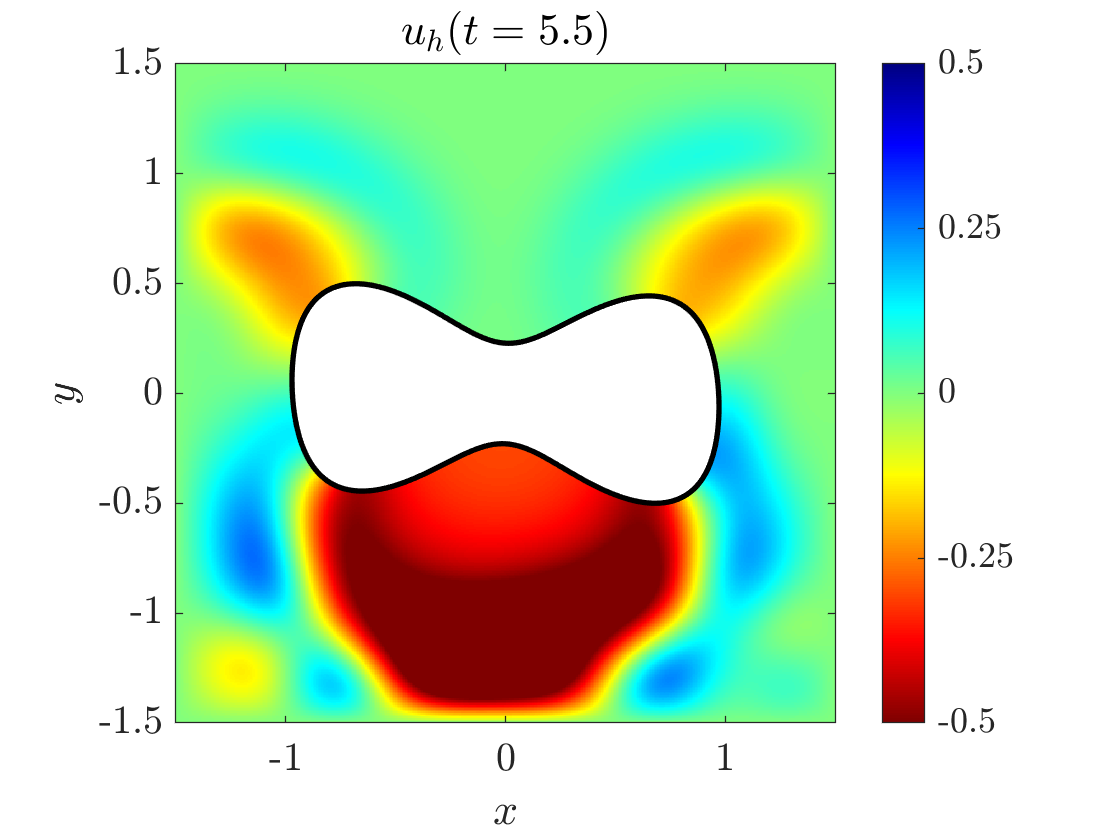}
\includegraphics[width=0.32\textwidth,trim={2cm 0.3cm 2cm 0.3cm},clip]{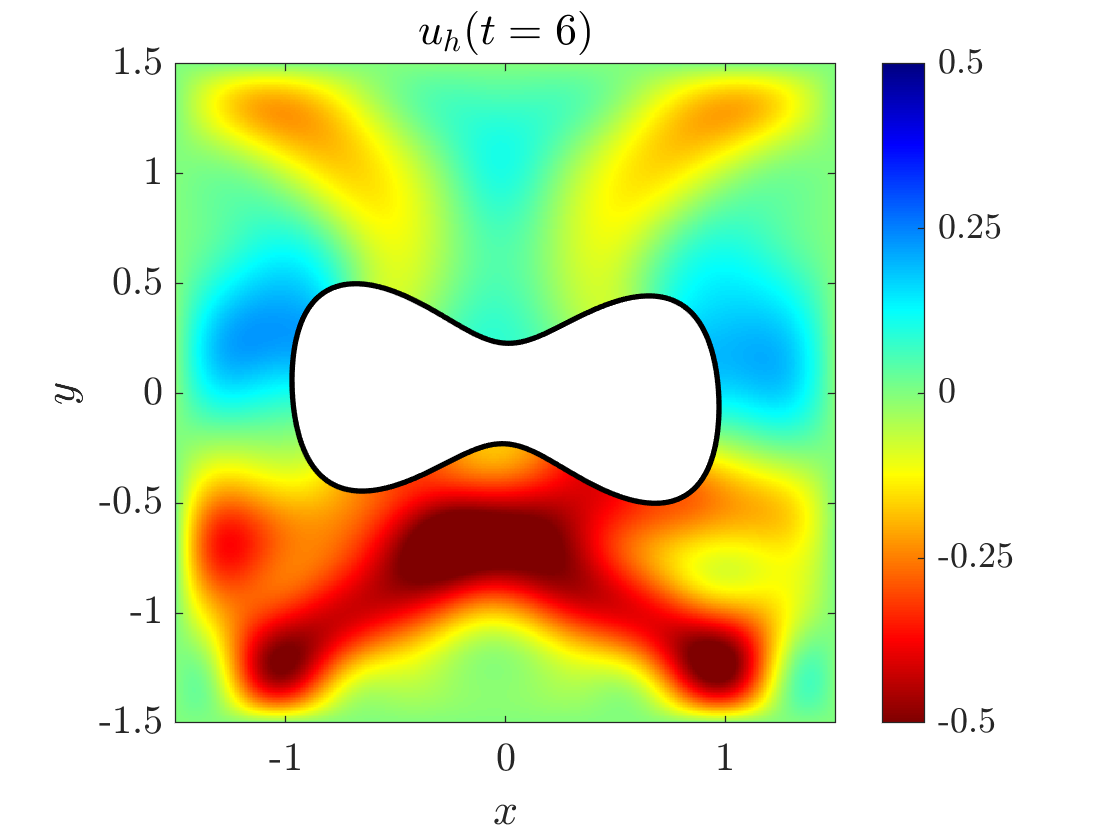}
\caption{Example \ref{ex2}: Snapshots of the numerical solution at $t=3.5, 4, 4.5, 5, 5.5$ and 6, computed with polynomial degrees $p=2$ and $q=1$ on a $400\times400$ Cartesian background mesh.}
\label{fig:ex2_snapshot}
\end{figure}

\begin{figure}[!ht]
\begin{center}
\includegraphics[width=0.4\textwidth]{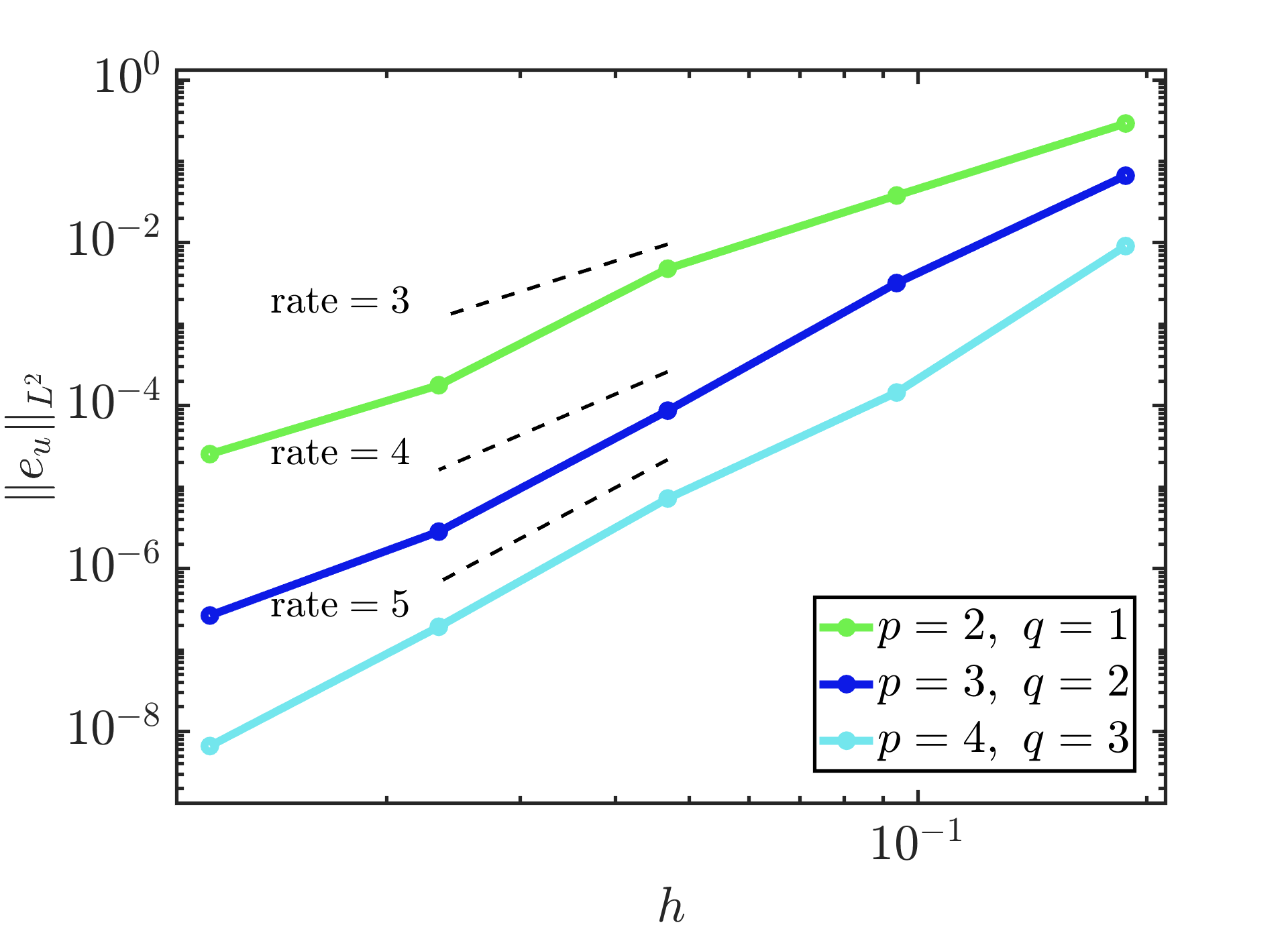}
\end{center}
\caption{Example \ref{ex2}: $L^2$-errors at $t = 5$ and observed convergence orders of $u_h$ under mesh refinement for different polynomial degree combinations $(p,q)$.}
\label{fig:ex2_accuracy}
\end{figure}

%%%%%%%%%%%%%%%%%%%%%%%%%%%%%%%%
\begin{example}\label{ex3} \textbf{Wave propagation around a triangular obstacle.}
\end{example}
We now consider the wave propagation in a square domain containing an immersed triangular obstacle. The computational domain is obtained by removing a triangle with vertices $(-1,0)$, $(1,0)$, and $(0,\sqrt{3})$ from the square $[-6,6]^2$. Homogeneous Dirichlet boundary conditions are imposed on both the outer boundary and the boundary of the triangular obstacle. The initial displacement and velocity are prescribed as
\begin{equation*}
    u(x,y,0) = \dfrac{1}{10}y(36-x^2-y^2)(\sqrt{3}x-y+\sqrt{3})(\sqrt{3}x+y-\sqrt{3})e^{-(x^2+y^2)},\quad u_t(x,y,0) = 0.
\end{equation*}

\begin{figure}[!ht]
\begin{center}
\includegraphics[width=3.0in]{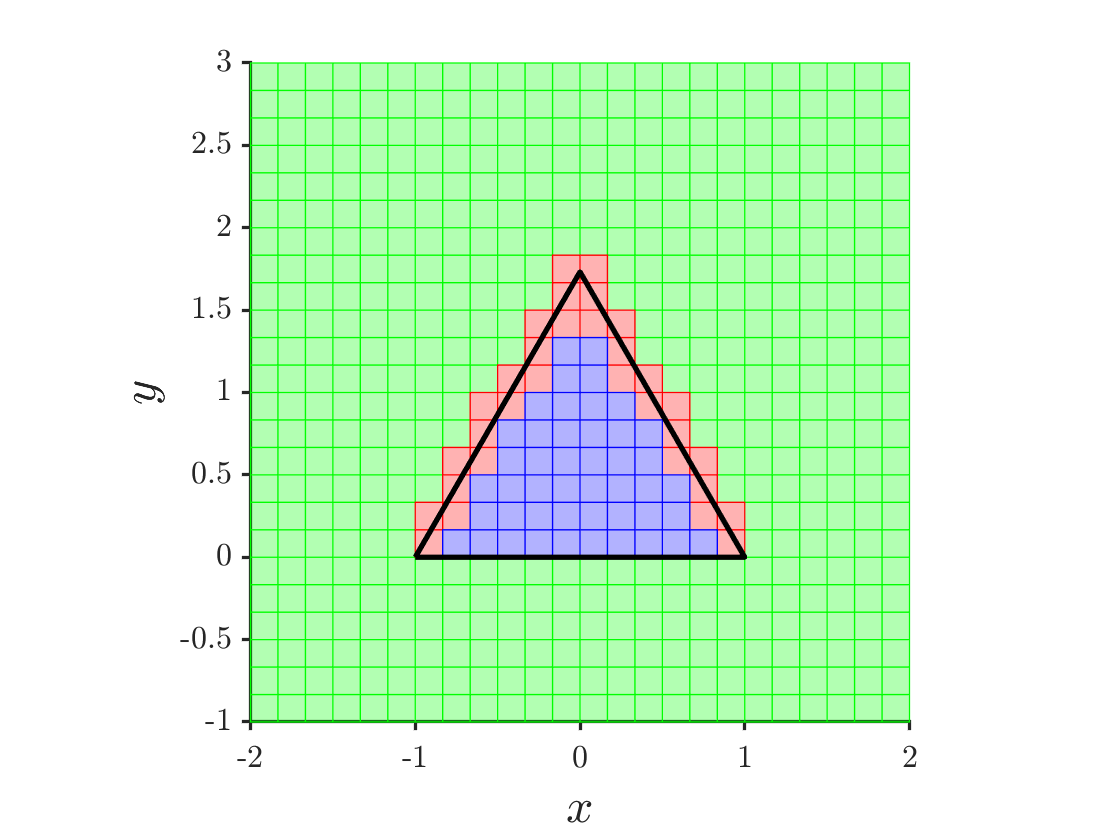}
\end{center}
\caption{Example \ref{ex3}: Illustration of a Cartesian background mesh containing a triangular obstacle, with the mesh elements cut by the non-grid-aligned boundary of the obstacle. The blue cells lie entirely within the computational domain, the red cells are cut by the non-grid-aligned boundary of the obstacle, and the green cells lie outside the computational
domain.}
\label{fig:ex3_background}
\end{figure}

The computational domain is discretized using a uniform Cartesian background mesh, with the triangular boundary cutting through the mesh elements. A representative mesh configuration is illustrated in Figure \ref{fig:ex3_background}. We choose polynomial degrees $p=2$ and $q=1$ and compute the numerical solution at $t=0, 0.2, 0.5, 0.8, 1.5,$ and $2$. The corresponding snapshots are presented in Figure \ref{fig:ex3_snapshot}, illustrating the evolution of the wave field around the triangular obstacle. In contrast to the preceding examples involving smooth immersed boundaries, the present configuration introduces sharp corners and non-smooth geometric features. The numerical results demonstrate the capability of the proposed CutEDG method to handle wave propagation in domains with such geometric complexities while retaining the flexibility of an unfitted Cartesian mesh.

\begin{figure}[!ht]
\centering 
\includegraphics[width=0.32\textwidth,trim={2cm 0.3cm 2cm 0.3cm},clip]{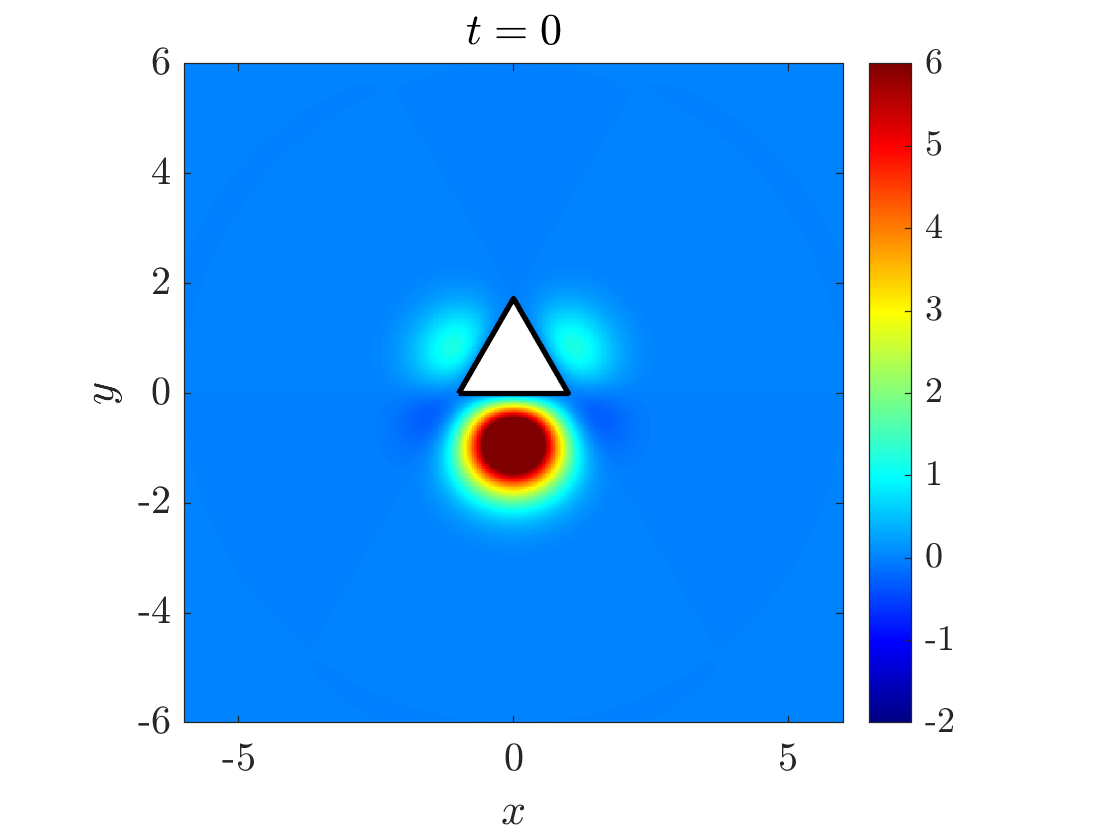}
\includegraphics[width=0.32\textwidth,trim={2cm 0.3cm 2cm 0.3cm},clip]{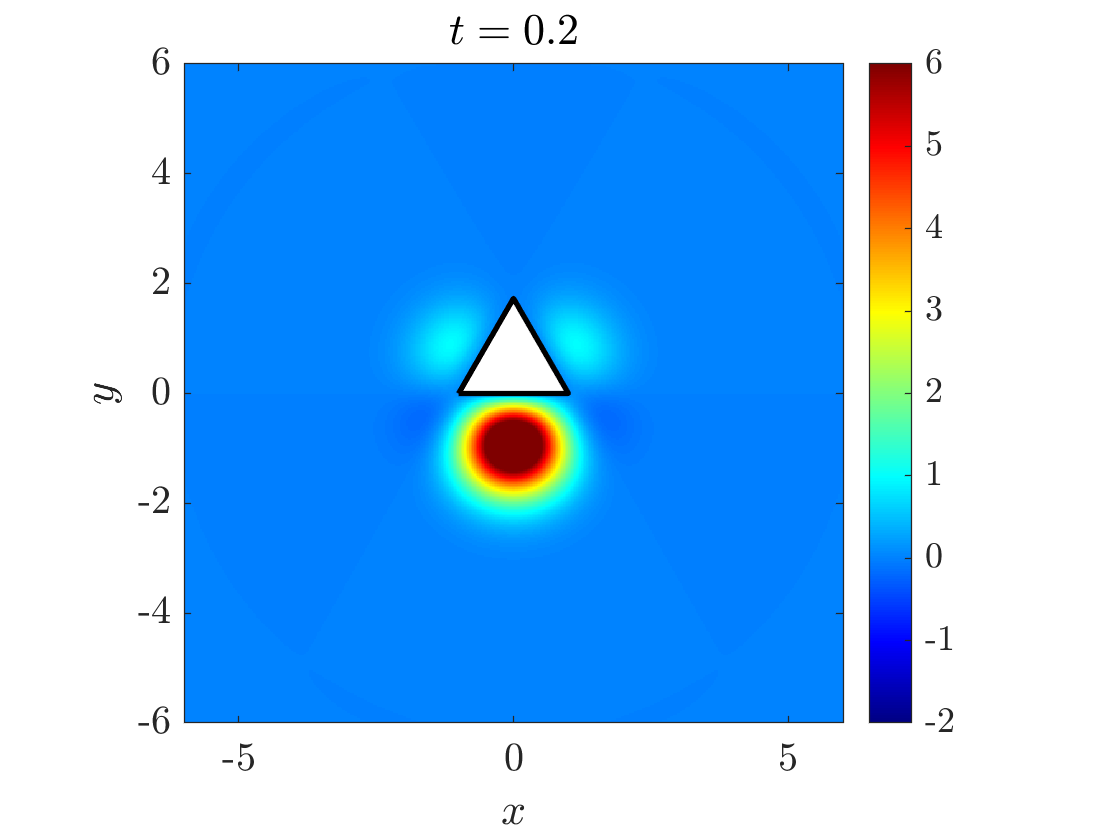}
\includegraphics[width=0.32\textwidth,trim={2cm 0.3cm 2cm 0.3cm},clip]{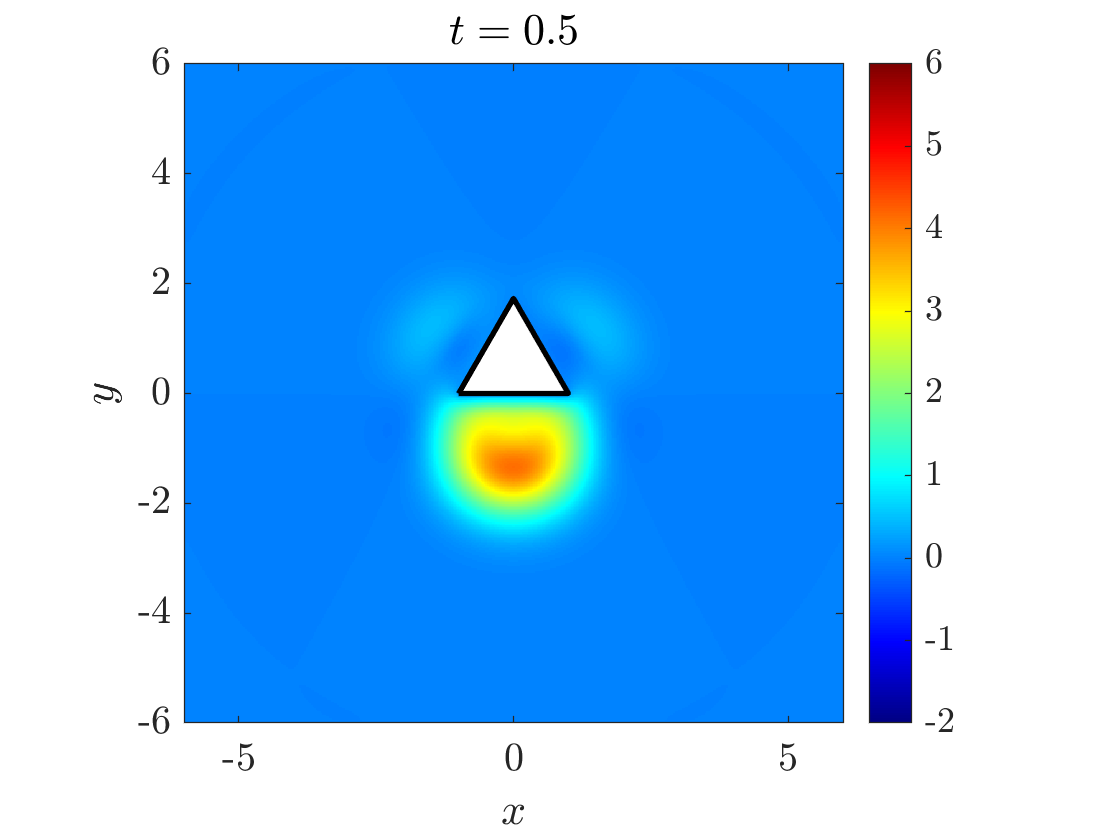}\\
\includegraphics[width=0.32\textwidth,trim={2cm 0.3cm 2cm 0.3cm},clip]{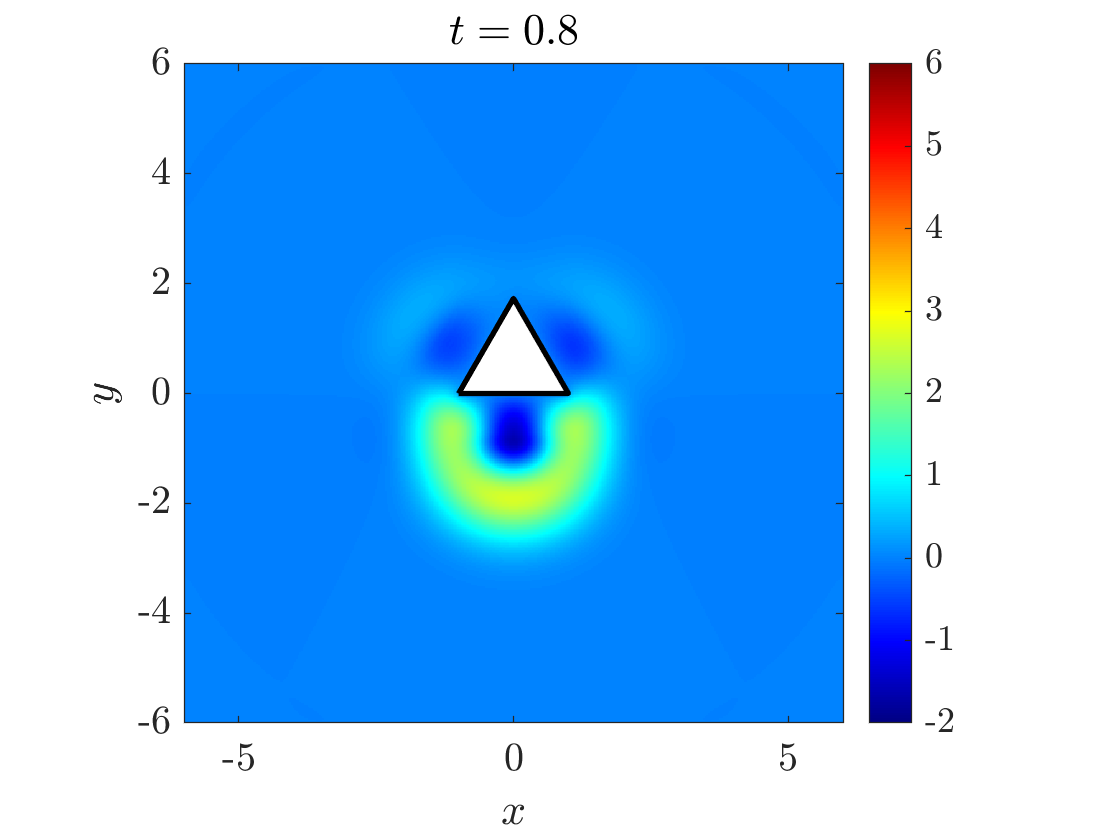}
\includegraphics[width=0.32\textwidth,trim={2cm 0.3cm 2cm 0.3cm},clip]{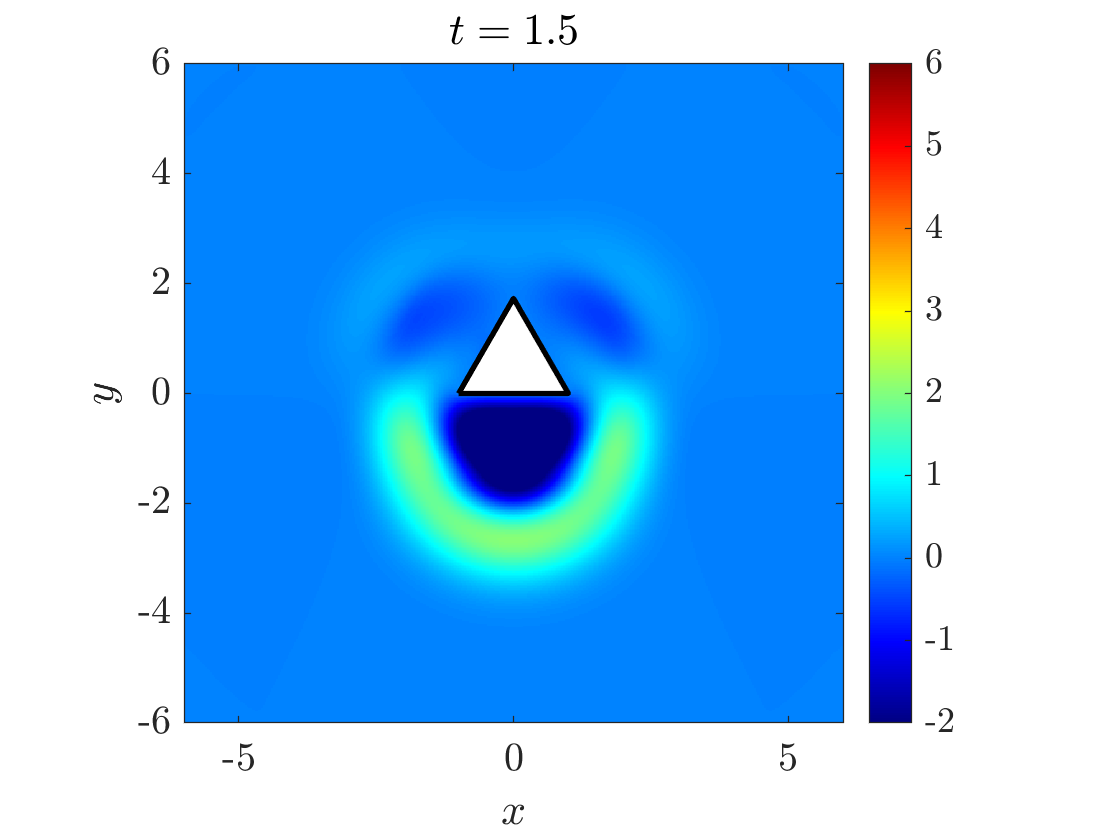}
\includegraphics[width=0.32\textwidth,trim={2cm 0.3cm 2cm 0.3cm},clip]{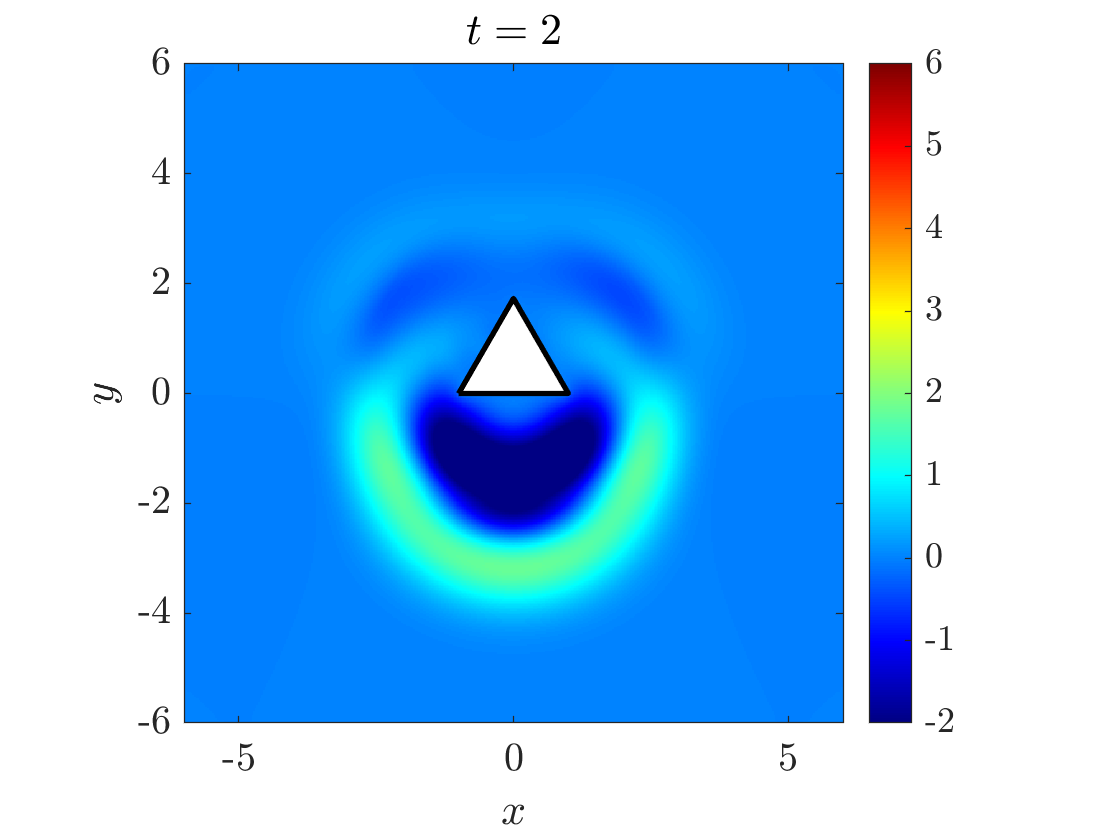}
\caption{Example \ref{ex3}: Snapshots of the numerical solution at $t=0, 0.2, 0.5, 0.8, 1.5,$ and $2$, computed on a $400\times400$ Cartesian background mesh.}
\label{fig:ex3_snapshot} 
\end{figure}

\begin{example}\label{ex4:circular-inclusion}
\textbf{Wave propagation across a circular material interface.}
\end{example}

We conclude the numerical study with a two-dimensional wave propagation problem in a heterogeneous medium containing a circular inclusion. In contrast to the preceding examples involving immersed boundaries, this problem features an internal material interface across which the wave speed is discontinuous. The purpose of this test is to assess the accuracy and convergence properties of the proposed CutEDG method for wave propagation across a curved, non-grid-aligned interface.

The computational domain is $\Omega=[-1.3,1.3]\times[-1.3,1.3]$, with the material interface defined by the unit circle $\Gamma=\{(x,y)\in\Omega:x^2+y^2=1\}$. The exterior and interior subdomains are given by $\Omega_1=\{(x,y)\in\Omega:x^2+y^2>1\}, \Omega_2=\{(x,y)\in\Omega:x^2+y^2<1\}$, with wave speeds $c_1=1$ and $c_2=1/2$, respectively. In each subdomain, the governing wave equation takes the form
\[
    u_{tt}=c_j^2\Delta u,
    \qquad (x,y)\in\Omega_j,\quad j=1,2,
\]
subject to the transmission conditions
\[
    [u]_\Gamma=0,
    \qquad
    [c^2\nabla u\cdot {\bf n}]_\Gamma=0,
\]
where $\bf n$ denotes a consistently oriented unit normal to $\Gamma$, and
$[\cdot]_\Gamma$ represents the jump across the interface $\Gamma$. On the outer boundary $\partial\Omega$, Dirichlet data are prescribed suing the analytical reference solution constructed below.

To obtain an analytical reference solution, we employ a Fourier--Bessel expansion of a time-harmonic wave field. Introducing polar coordinates $r=\sqrt{x^2+y^2}, \theta=\operatorname{atan2}(y,x)$, we set the angular frequency to $\omega = 2\pi$, giving the wavenumbers $k_1={\omega}/{c_1}=2\pi, k_2={\omega}/{c_2}=4\pi$.
The complex-valued time-harmonic solution in each subdomain is expressed as
\[
    \widetilde u_j(x,y,t)
    =\Phi_j(r,\theta)e^{\mathrm{i}\omega t},
    \qquad j=1,2.
\]
In the exterior subdomain, the spatial amplitude is represented as the superposition of incident and scattered waves,
\[
\Phi_1(r,\theta)
= B_0(k_1r)+A_0H_0^{(2)}(k_1r) +\sum_{n=1}^{100}
\left[I_nB_n(k_1r)+A_nH_n^{(2)}(k_1r)\right]
\cos(n\theta),
\]
where $I_n=2(-\mathrm{i})^n$ for $n\geq1$. The transmitted field inside the circular inclusion is represented by
\[
\Phi_2(r,\theta)
=C_0B_0(k_2r)
+\sum_{n=1}^{100}C_nB_n(k_2r)\cos(n\theta).
\]
Here, $B_n$ and $H_n^{(2)}$ denote the Bessel function of the first kind
and the Hankel function of the second kind, both of order $n$, respectively. 

The expansion coefficients $A_n$ and $C_n$ are determined by enforcing the transmission conditions at $r=1$. Setting $I_0=1$ and using primes to denote differentiation with respect to the function argument, we obtain
\[
A_n=I_n
\frac{
c_2^2 k_2B_n(k_1)B_n'(k_2)
-c_1^2 k_1B_n'(k_1)B_n(k_2)
}{
c_1^2 k_1H_n^{(2)\prime}(k_1)B_n(k_2)
-c_2^2 k_2H_n^{(2)}(k_1)B_n'(k_2)
},\qquad C_n=
\frac{I_nB_n(k_1)+A_nH_n^{(2)}(k_1)}
{B_n(k_2)},
\]
for $n = 0,...,100$. The zeroth-order mode is evaluated separately in the implementation, while the higher-order modes are retained up to $n=100$. The resulting truncated Fourier--Bessel expansion serves as the analytical reference solution for the numerical experiments.

The corresponding real-valued solution is defined by
\[
u_{\mathrm{ex}}(x,y,t)=
\begin{cases}
\operatorname{Re}\!\left\{
\Phi_1(r,\theta)e^{\mathrm{i}\omega t}
\right\}, &(x,y)\in\Omega_1,\\[2mm]
\operatorname{Re}\!\left\{
\Phi_2(r,\theta)e^{\mathrm{i}\omega t}
\right\}, &(x,y)\in\Omega_2.
\end{cases}
\]
The initial displacement and velocity are obtained from this reference
solution. On the outer boundary, we also impose $u = u_{\rm ex}$. 
This example is adapted from the SBP interface treatment of \cite{wang2016high}, where the domain must be partitioned into multiple blocks aligned with the interface. No such decomposition is needed here: the circular interface is embedded directly on a single background mesh.

\begin{figure}[htbp]
\includegraphics[width=0.32\textwidth,trim={1cm 0.3cm 2cm 0.3cm},clip]{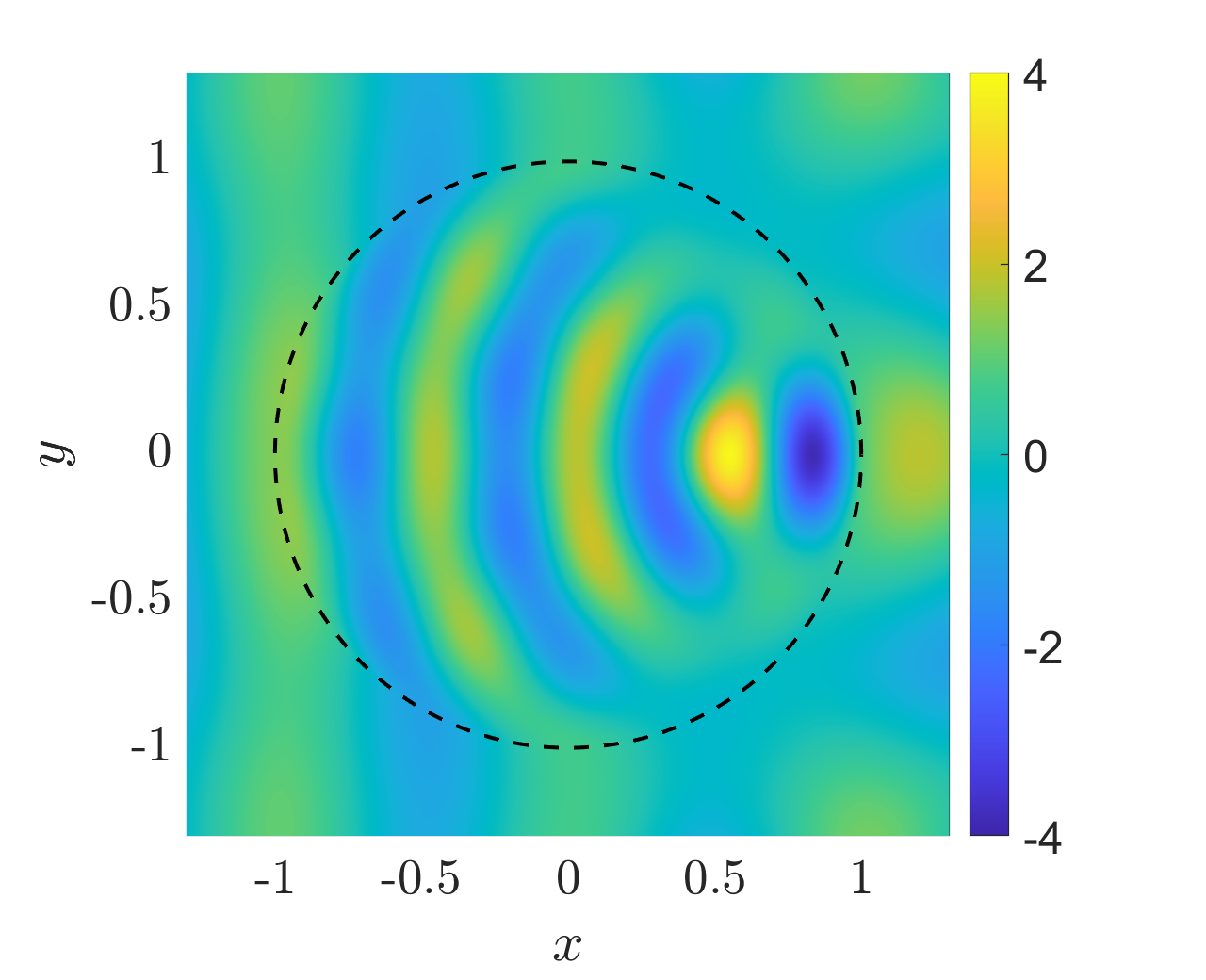}
\includegraphics[width=0.32\textwidth,trim={1cm 0.3cm 2cm 0.3cm},clip]{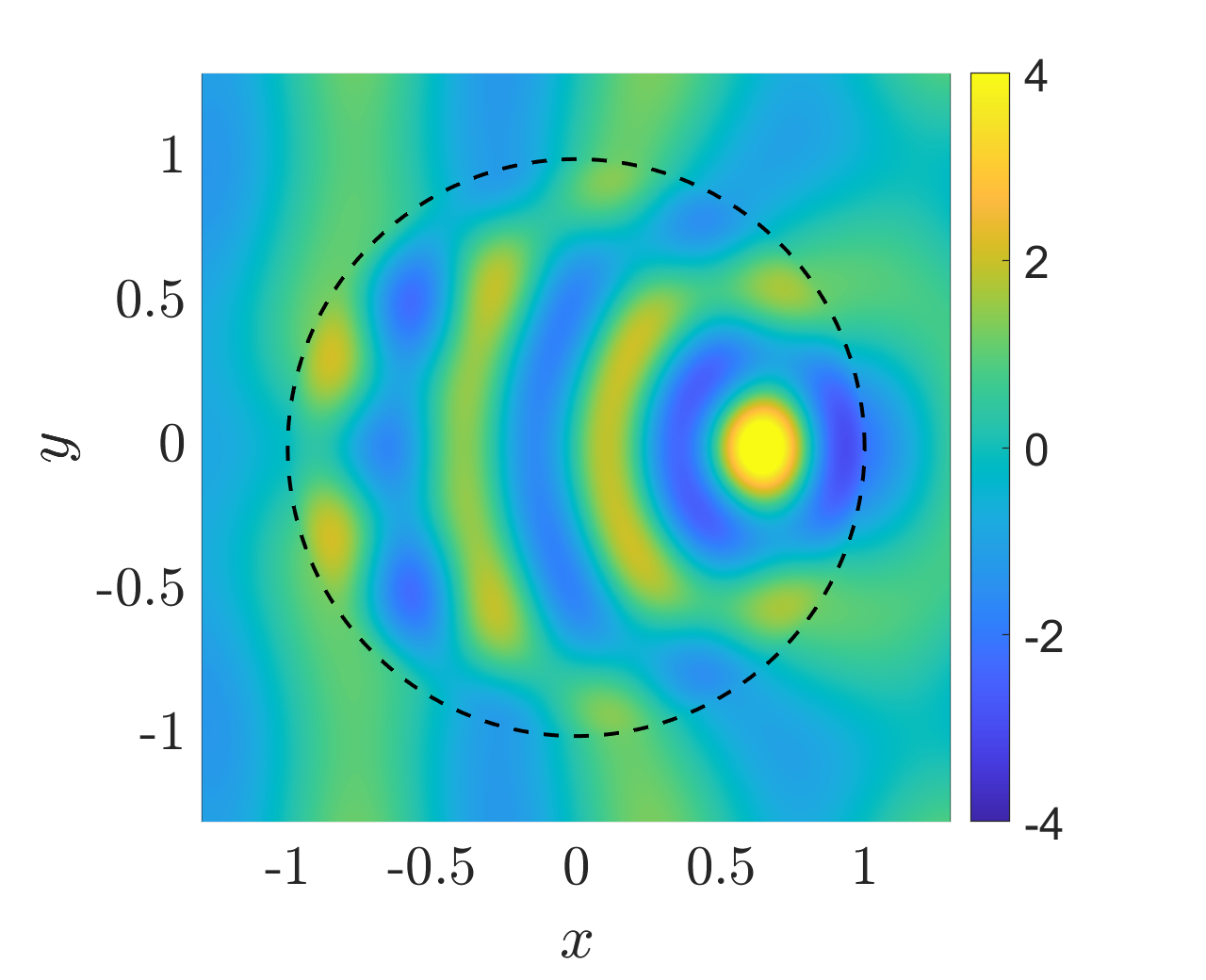}
\includegraphics[width=0.32\textwidth,trim={1cm 0.3cm 2cm 0.3cm},clip]{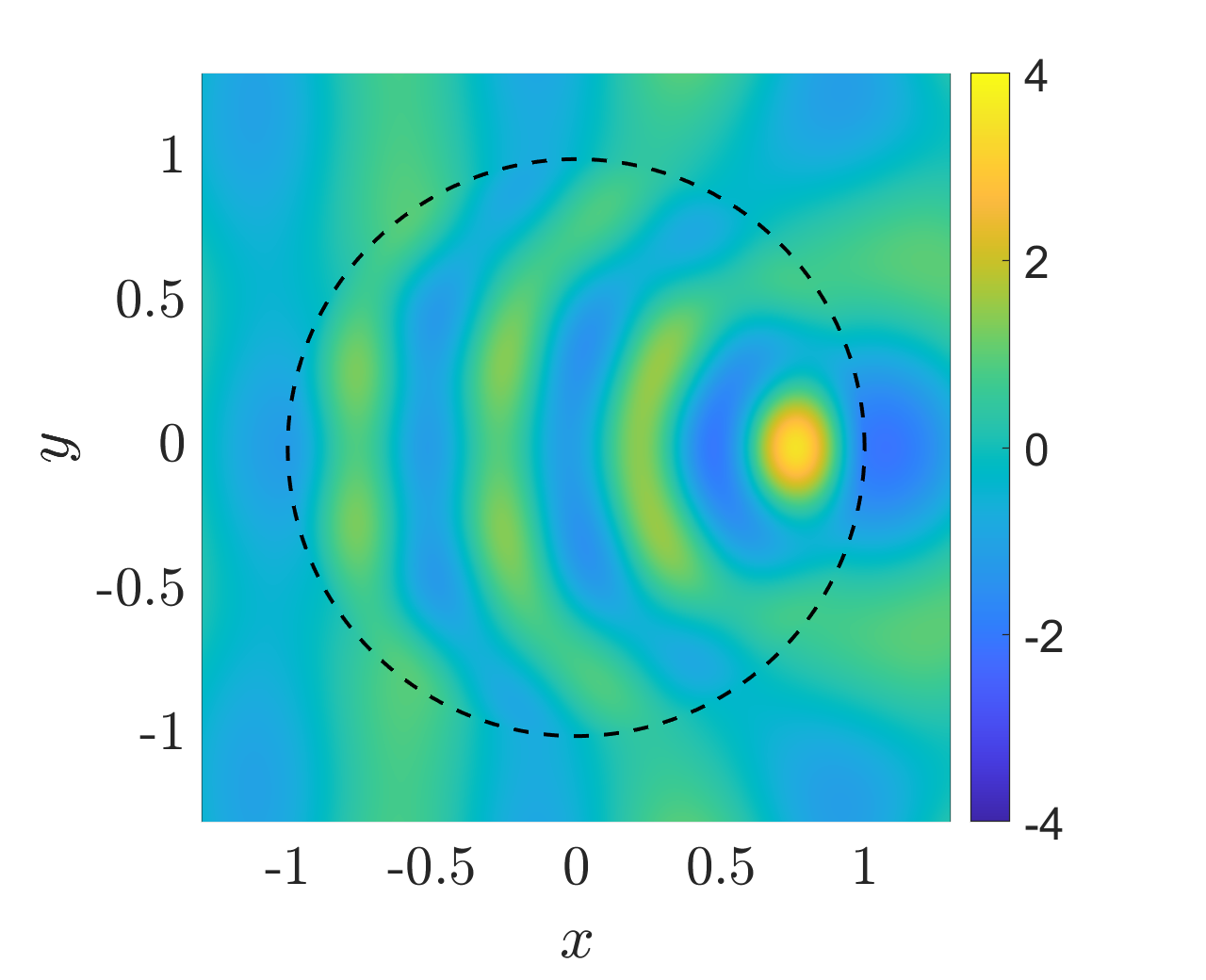}\\
\includegraphics[width=0.32\textwidth,trim={1cm 0.3cm 2cm 0.3cm},clip]{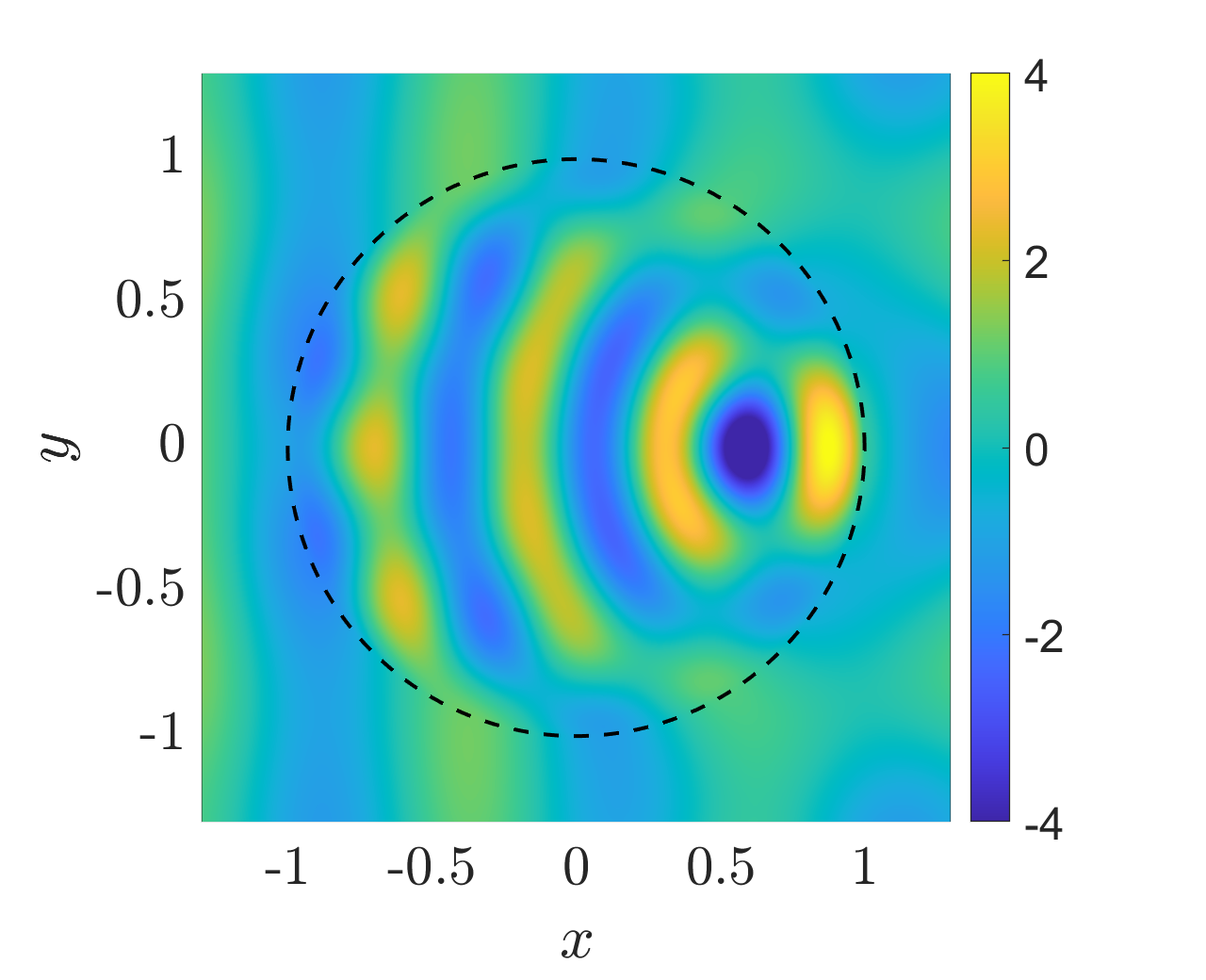}
\includegraphics[width=0.32\textwidth,trim={1cm 0.3cm 2cm 0.3cm},clip]{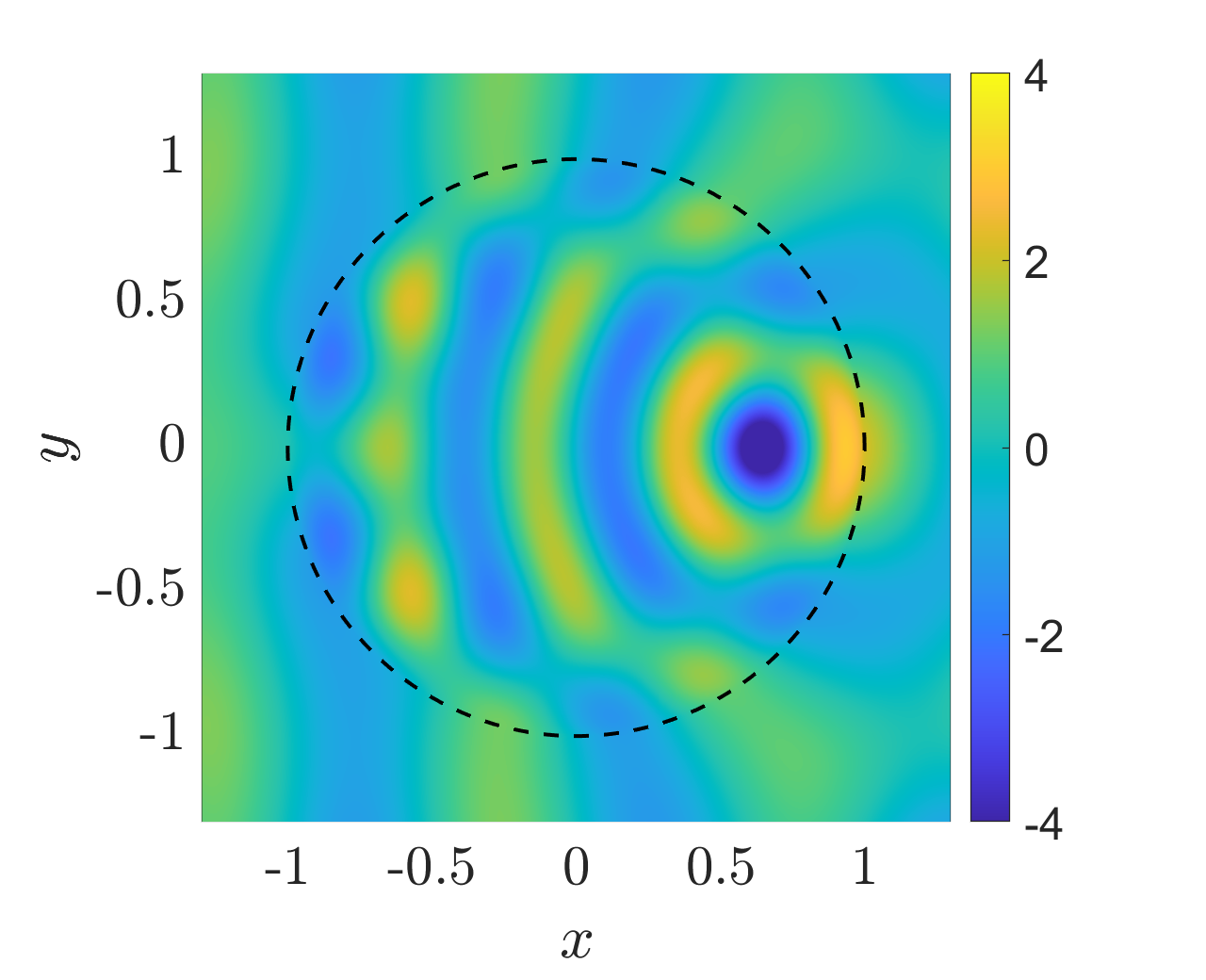}
\includegraphics[width=0.32\textwidth,trim={1cm 0.3cm 2cm 0.3cm},clip]{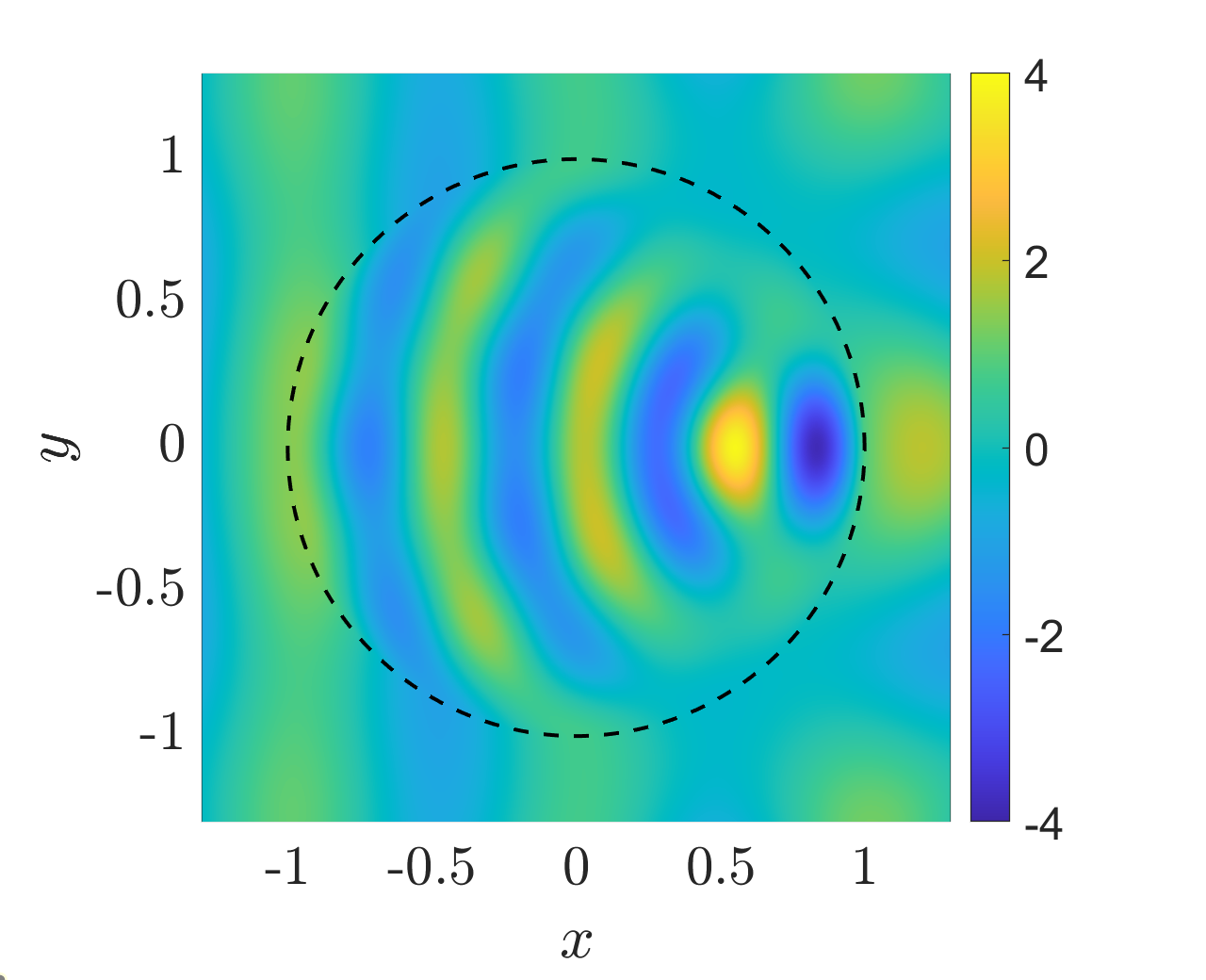}
\caption{Example \ref{ex4:circular-inclusion}: Snapshots (top view) of the numerical solution $u_h$ for the circular material interface problem at $t = 0, 0.25, 0.4, 0.6, 0.75,$ and $1$, arranged from left to right and top to bottom. The dashed circle in each panel marks the material interface $\Gamma={(x,y)^2+y^2=1}$.}
    \label{fig:ex4_circular-inclusion-solutions}
\end{figure}

Figure \ref{fig:ex4_circular-inclusion-solutions} presents snapshots of the numerical solution $u_h$ at $t=0, 0.25, 0.4, 0.6, 0.75$, and $1$, computed using polynomial degrees $p=2$ and $q=1$ for the scalar and auxiliary variables, respectively. The snapshots illustrate the spatial structure of the wave field arising from reflection and transmission at the circular material interface. Owing to the time-harmonic nature of the reference solution, the wave field evolves periodically, with different snapshots corresponding to different phases of the oscillation. The numerical results illustrate the ability of the proposed CutEDG method to capture the wave behavior across a curved interface with discontinuous material properties, while employing a Cartesian background mesh that does not conform to the interface.

To further investigate the spatial accuracy of the method, we perform a uniform mesh refinement study for polynomial degrees $p=2,3,4$, with $q=p-1$. The Cartesian background mesh is refined from $N_x=N_y=40$ to $N_x=N_y=400$. At the final time $T=1$, the numerical solution is compared with the analytical reference solution, and the $L^2$-error is evaluated over the entire computational domain. The resulting errors and convergence rates are presented in Figure \ref{fig:ex4_accuracy},
Nevertheless, the numerical errors decrease systematically under mesh refinement, with the observed optimal convergence rate $p+1$ for all polynomial degrees considered. These results demonstrate that the proposed CutEDG method retains its high-order accuracy in the presence of discontinuous wave speeds and curved, non-grid-aligned material interfaces.

\begin{figure}[!ht]
\begin{center}
\includegraphics[width=0.4\textwidth]{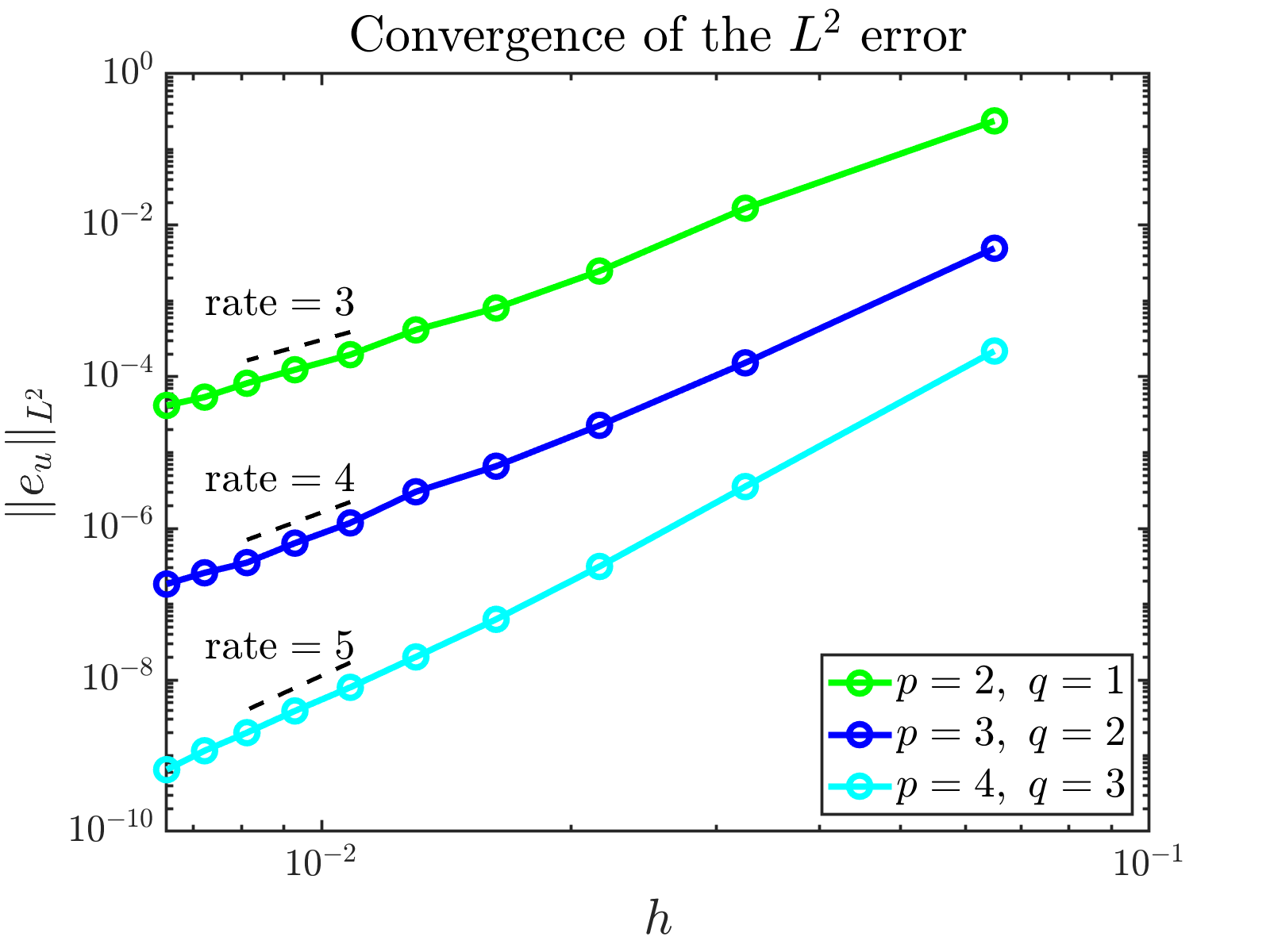}
\end{center}
\caption{ Example \ref{ex4:circular-inclusion}: Convergence of the numerical errors under mesh refinement for different polynomial approximation spaces.}
\label{fig:ex4_accuracy}
\end{figure}

%%%%%%%%%%%%%%%%%%%%%%%%%%%%%%%%%%%
\section{Conclusion}\label{sec:concluding remarks}

In this work, we developed a high-order CutEDG method for the wave equation on unfitted meshes. By introducing ghost-penalty stabilization directly into the local bilinear forms defining the discrete EDG energy, the method retains the original flux structure while restoring robustness with respect to arbitrarily small cut cells. We established semidiscrete energy stability and high-order error estimates, and extended the formulation to stationary interface problems with discontinuous wave speeds. Numerical experiments confirmed the expected convergence rates, cut-independent conditioning, and an explicit time-step restriction governed by the background mesh size rather than the smallest physical cut cell. These results demonstrate the potential of CutEDG as a robust and accurate framework for high-order wave propagation in geometrically complex media.

\bibliographystyle{abbrv}
\bibliography{refs}

\end{document}